\documentclass[aop,preprint]{imsart}
\RequirePackage[OT1]{fontenc}
\RequirePackage{amsthm,amsmath}
\RequirePackage[numbers]{natbib}
\RequirePackage[colorlinks,citecolor=blue,urlcolor=blue]{hyperref}

\usepackage[utf8]{inputenc}
\usepackage{amsfonts, amsmath, amssymb, amsthm, bbm, color, enumerate, graphicx, mathtools, tikz, hyperref, relsize}

\newcommand{\DD}{\mathcal{D}}

\newcommand{\yy}{\mathbf{y}}			

\newcommand{\xx}{\mathbf{x}}
\newcommand{\XX}{\mathbf{X}}

\newcommand{\ii}{\mathbf{i}}
\newcommand{\jj}{\mathbf{j}}
\newcommand{\kk}{\mathbf{k}}
\newcommand{\YY}{\mathbf{Y}}
\newcommand{\zero}{\mathbf{0}}

\newcommand{\Vt}{\mathcal{V}}
\newcommand{\A}{\mathcal{A}}
\newcommand{\Vy}{\mathbf{V}}
\newcommand{\Et}{\mathcal{E}}
\newcommand{\Ey}{\mathbf{E}}
\newcommand{\T}{\mathbb{T}}

\newcommand{\sg}{\gamma^{(1)}}
\newcommand{\sga}{\gamma^{*}}
\newcommand{\sgt}{\gamma^{(2)}}
\newcommand{\wt}{\widetilde}
\newcommand{\wh}{\widehat}
\newcommand{\ol}{\overline}
\newcommand{\hl}{\widehat\lambda}
\newcommand{\hld}{\widehat\lambda^{(d)}}

\allowdisplaybreaks

\newtheorem{theorem}{Theorem}[section]
\newtheorem{corollary}[theorem]{Corollary}
\newtheorem{lemma}[theorem]{Lemma}
\newtheorem{proposition}[theorem]{Proposition}

\theoremstyle{definition}

\newtheorem{remark}[theorem]{Remark}

\let\plainqed\qedsymbol
\newcommand{\claimqed}{$\lrcorner$}

\hypersetup{
 colorlinks=true,
 linkcolor=blue,          
 citecolor=red,       
 filecolor=red,   
 urlcolor=red, 
 pdftitle={},
 pdfauthor={},
 pdfsubject={},
 pdfkeywords={},
 linktocpage=true
}

\usepackage{fancyhdr}
 
\usepackage{natbib}

\usepackage{amssymb}
\usepackage{amsmath}
\usepackage{amsthm}

\newcommand{{\LPC}}{\textbf{LPC}}

\newcommand{\x}{\mathbf{x}}
\newcommand{\y}{\mathbf{y}}
\newcommand{\Y}{\mathbf{Y}}

\newcommand{\0}{\mathbf{0}}

\newcommand{\LL}{\mathcal{L}}
\newcommand{\EE}{\mathcal{E}}

\newcommand{\eps}{\varepsilon}

\newcommand{\by}{\overline{y}}
\newcommand{\bY}{\overline{Y}}

\numberwithin{equation}{section}

\begin{document}

\begin{frontmatter}

\title{Rates of convergence to equilibrium for Potlatch and Smoothing processes}
\runtitle{Potlatch and Smoothing Processes}
\author{\fnms{Sayan} \snm{Banerjee}\ead[label=e1]{sayan@email.unc.edu}}
\thanks{SB's research was supported in part by a Junior Faculty Development Award by UNC, Chapel Hill. }
\address{Department of Statistics \\and Operations Research\\
353 Hanes Hall CB \#3260\\
University of North Carolina\\
Chapel Hill, NC 27599\\
\printead{e1}}
\affiliation{University of North Carolina, Chapel Hill}
\and
\author{\fnms{Krzysztof} \snm{Burdzy}\ead[label=e2]{burdzy@uw.edu}}
\thanks{KB's research was supported in part by Simons Foundation Grant 506732. }
\address{Department of Mathematics,\\
Box 354350,\\
University of Washington,\\
Seattle WA 98195\\
\printead{e2}}
\affiliation{University of Washington, Seattle}
\runauthor{Banerjee and Burdzy}

\begin{abstract}
We analyze the local and global smoothing rates of the smoothing process and obtain convergence rates to stationarity for the dual process known as the potlatch process. For general finite graphs, we connect the smoothing and convergence rates to the spectral gap of the associated Markov chain. We perform a more detailed analysis of these processes on the torus. Polynomial corrections to the smoothing rates are obtained. They show that local smoothing happens faster than global smoothing. These polynomial rates translate to rates of convergence to stationarity in $L^2$-Wasserstein distance for the potlatch process on $\mathbb{Z}^d$.
\end{abstract}


\end{frontmatter}

\section{Introduction}
We investigate the rate of convergence to equilibrium for the potlatch and smoothing processes. The  potlatch process can be described as a random mass redistribution process on a set of $n$ sites $\mathcal{I} = \{1,2,\dots,n\}$ ($n$ possibly infinite) where each site is activated according to an independent copy of a Poisson process with unit intensity. When a site activates, it redistributes its mass to all the sites in proportion to the transition kernel of a Markov chain with state space $\mathcal{I}$. For this article, we will assume that the Markov chain is reversible with respect to a probability measure $\pi$ on $\mathcal{I}$, although it is not necessary for the construction of the process. This process, under the name potlatch process, was introduced on the lattice $\mathbb{Z}^d$ ($d \ge 1$) in \cite{liggett1981ergodic} where the Markov transition kernel governing redistribution of mass was that of a homogeneous random walk on $\mathbb{Z}^d$ (see \cite{holley1981generalized} for a natural generalization of the model). The process was constructed and it was shown that there exists a unique translation invariant stationary measure $\nu_{\rho}$ for the dynamics with given expected mass per site $\rho$. It was also shown that the process started from any ergodic initial distribution (with respect to translations on $\mathbb{Z}^d$) with expected mass $\rho$ per site converges weakly to the stationary measure $\nu_{\rho}$. Although $\nu_{\rho}$ is far from explicit, it was shown (see \cite[Thm. 1.9]{liggett1981ergodic}) that the stationary means, variances and covariances for $\nu_{\rho}$ can be explicitly computed. It was observed in \cite{liggett1981ergodic} that the potlatch process has a nice dual representation called the smoothing process. 
For the smoothing process, whenever a site activates, its mass gets updated to a weighted average of the current mass at all the sites, with weights governed by the same transition kernel (see Section \ref{mainres} for a precise description). We note here that the potlatch and smoothing processes fall under the broad class of `linear systems'. See \cite[Ch. IX]{Lig85} for an accessible presentation of this material as a book chapter. Numerous aspects of such systems like localization phenomena \cite{nagahata2009localization}, central limit theorems \cite{nagahata2009central,nagahata2010note} and fixed point analysis \cite{durrett1983fixed}, to name a few, have been subsequently studied.

A related model, called the random average process has been studied in \cite{ferrari1998fluctuations,rajesh2001exact,balazs2006random} (see \cite{aldous2012lecture} for a variant of this model). More recently, a variety of models related to the smoothing process on general graphs have been used to model opinion dynamics on social networks \cite{acemouglu2013opinion,frasca2013gossips,yildiz2011discrete,shi2016evolution,ghaderi2014opinion,shah2009gossip} where the mass at each site (thought of as an agent) corresponds to the current opinion (equivalently, belief or knowledge) of the agent. Whenever the clock rings at a site, the agent interacts with other members in the community and updates her opinion according to a weighted combination of opinions of all the members. The weights model a variety of aspects like stubbornness, geographical proximity, social prominence, etc. We note here that the clock rings associated to opinion dynamics models are usually edge-based (describing times when the agents adjacent to the edge communicate) although mathematically, the edge-based dynamics and existing results on them are very similar (see \cite{acemouglu2013opinion}) to those for the smoothing process. We will exclusively use the site-based clock rings in this article.

In the context of the smoothing process (equivalently, opinion dynamics) on large finite graphs, some natural questions arise: (i) How long does it take for the smoothing process to stabilize (approach a global random equilibrium) and how does it depend on the geometry of the graph? (ii) Does the profile smooth out locally at a faster rate than the convergence rate to global equilibrium? (iii) Does the average opinion converge faster than the collective set of opinions? For the dual potlatch  process, natural questions concern rates of convergence to stationarity.

There are numerous existing results on rates of convergence to equilibrium for related interacting particle systems (both in finite and infinite volume) like the simple exclusion process \cite{ferrari1999rate,nagahata2012lower}, the zero range process \cite{janvresse1999relaxation}, the interchange process \cite{caputo2010proof,forsstrom2017spectrum}, the averaging process \cite{aldous2012lecture}, etc., to name a few. However, most of these models possess an explicit collection of stationary measures and the dynamics is reversible with respect to those. Not much is known about the stationary measure of the potlatch process beyond the first two moments and, moreover, the dynamics is not reversible with respect to it. Moreover, the state space of the mass at each site is non-compact. The dual smoothing process, which turns out to be technically slightly easier to analyze (as convergence issues can be addressed without referring to an unknown stationary measure), is not a conservative process unlike the above systems in the sense that the total mass is not conserved. In fact, the total/average mass converges asymptotically to a random variable whose value depends non-trivially on the order in which the clocks ring at different sites. These aspects make this model non-standard and technically challenging to analyze.

Recently, \cite{BBPS} studied the potlatch process on finite graphs  under the name of meteor process and \cite{burdzymeteor} revisited the lattice case. Among other things, \cite{BBPS} investigated the rate of convergence of the potlatch process on a finite graph in $L^1$-Wasserstein distance. For the discrete $d$-dimensional torus $\T_n^d := (\mathbb{Z}/n\mathbb{Z})^d$ they obtained an $O(n^2)$ bound on the relaxation time (see \cite[Thm. 3.6]{BBPS}). They further conjectured that for general finite graphs, the relaxation time should be related to the mixing time (and hence spectral gap) of the random walk on the graph. In this article, we address this conjecture. We show that for the general smoothing processes constructed from reversible transition kernels described above, the mass profile approaches a random global equilibrium (quantified by the variance functional) at an exponential rate that can be obtained explicitly in terms of the spectral gap of the Markov chain associated to the reversible kernel (Theorem \ref{global}). The proof of Theorem \ref{global} can also be used to give a quantitative bound on the local smoothing rate of the mass profile in terms of an energy functional (Corollary \ref{local}). By duality, Theorem \ref{global} directly translates to rates of convergence for the potlatch process in $L^2$-Wasserstein distance (Theorem \ref{meteorconv}). 

Next, we take a closer look at the smoothing process on the torus (with associated kernel being that of the simple random walk) starting from unit mass at zero and zero mass elsewhere. We use martingale analysis and spectral theory to give `almost matching' upper and lower bounds on the rate of decay of the expectation of the variance and energy functionals which respectively quantify the global and local smoothness of the evolving mass profile (Theorem \ref{wass}). These bounds not only capture the exponential convergence rate, but also the polynomial decay term before the exponential term. These polynomial terms are essentially independent of the size of the torus and capture the smoothing rates before time $t \le n^2$ when the exponential decay rate does not take effect. Moreover, the explicit difference in orders of these polynomial rates directly implies that local smoothing happens faster than global smoothing.  These explicit quantitative bounds rigorously establishing the difference between local and global smoothing rates address an open problem in \cite[Sec. 3.3 (i)]{aldous2012lecture} (which investigated the averaging process) in the context of the smoothing process (see the discussion after Proposition \ref{d19.10}). Theorem \ref{wass} can be used to obtain an upper bound on the rate at which the average mass of all the sites approaches the equilibrium mass (which is necessarily the same across sites) starting from an arbitrary initial mass profile (Theorem \ref{avgquick}). In particular, this bound shows that when the smoothing process starts from positive mass at a site and zero mass elsewhere, the average mass approaches equilibrium at a faster rate than the global convergence rate to equilibrium (see Remark \ref{avgrem}). Using duality, Theorem \ref{wass} can be used to give improved bounds on the rate of convergence to stationarity for the potlatch process started from certain initial configurations (Theorem \ref{metone}). These bounds contain the crucial polynomial terms (independent of $n$) that, by letting the size of the torus go to infinity, can be used to obtain polynomial rates of convergence to equilibrium in $L^2$-Wasserstein distance for local statistics of the potlatch process on $\mathbb{Z}^d$ (Theorem \ref{meteorzd}).

In this article, we introduce a number of new techniques that can potentially be used to address convergence rates to equilibrium for a class of non-reversible stochastic systems with non-compact state space. In particular, the use of appropriately chosen martingales to derive a renewal representation of the derivative of the energy functional as a convolution series (Lemma \ref{mom2delta}) and novel estimates for the associated convolutions (Lemma \ref{enderlem}) are more general than the considered models.

It would be interesting to investigate whether one can obtain lower bounds on rates of convergence to equilibrium for the smoothing and potlatch processes on general graphs in terms of the spectral gap of the associated Markov chain to complement the results in Theorems \ref{global} and \ref{meteorconv}. Another interesting direction of future research would be to obtain quantitative rates of convergence to equilibrium for the potlatch process on $\mathbb{Z}^d$ analogous to Theorem \ref{meteorzd} for more general (not necessarily translation invariant) initial configurations. We note here that results on weak convergence to stationarity (without rates) have been obtained in \cite{roussignol1980processus,cox2000convergence} for more general initial configurations.

Section \ref{mainres} defines the models rigorously and states the main results. 
Section \ref{prelim} contains a review of elementary spectral theory.
Theorems \ref{global} and \ref{meteorconv} are proved in Section \ref{revker}. Theorems \ref{wass}, \ref{avgquick}, \ref{metone} and \ref{meteorzd} are proved in Section \ref{torus}.

\subsection{Acknowledgment}
We are grateful to the anonymous referee for very helpful advice.

\subsection{Acknowledgment of priority} The second author takes this opportunity to acknowledge the priority of results on the potlatch process obtained originally in  \cite{liggett1981ergodic,holley1981generalized}. Specifically,
\cite[Thm. 5.1]{BBPS},  \cite[Thms. 3.1, 3.4, 4.1]{burdzymeteor} and a few less significant results had been proved  earlier in \cite{liggett1981ergodic,holley1981generalized}.
Sections 4, 6 and 7 of \cite{BBPS} and Sections 5 and 6 of \cite{burdzymeteor} contain results that had not been proved earlier, to our best knowledge.

\section{Model definition and main results}\label{mainres}

We start with the definitions of potlatch and smoothing processes. The constructions of these processes can be found, for example, in \cite[Ch. IX]{Lig85}.

Consider a family of ``agents'' or sites $\mathcal{I} = \{1,2,\dots,n\}$, where $n$ can be finite or infinite. In the $n=\infty$ case, the meaning of notation such as  $(X_i(t))_{1\le i \le n}$ should be clear from context.

Consider a transition matrix $P = (p_{ij})_{1\le i,j \le n}$ which is irreducible, aperiodic and reversible with respect to a probability vector $\pi = (\pi_i)_{1 \le i \le n}$, namely $\pi_i p_{ij} = \pi_jp_{ji}$ for all $1 \le i,j \le n$. We define the potlatch process on  $\mathcal{I}$ corresponding to $P$  as a random mass distribution $\{\XX(t) = (X_i(t))_{1\le i \le n}\}_{t \ge 0}$ that evolves in the following way: Start with any (possibly random) initial configuration $\XX(0) \ge 0$. Each agent has a Poisson clock attached to it, independent of $\XX(0)$ and independent of all other clocks. If the Poisson clock at $i$ rings at time $t$, the mass at $i$ gets updated to $X_i(t) = p_{ii}X_i(t-)$ and for $j \neq i$, the mass at $j$ gets updated to $X_j(t) = X_j(t-) + p_{ij}X_i(t-)$. Note that the total mass is conserved in this dynamics.

The dual process which we denote by $\{\Y(t) = (Y_i(t))_{1\le i \le n}\}_{t \ge 0}$, called the smoothing process, can be described as follows: Start from any (possibly random) initial mass distribution $\Y(0) = (Y_j(0))_{1 \le j \le n}$ with $\Y(0) \ge 0$. Attach a Poisson process to each site. 
Assume that all these Poisson processes are jointly independent and they are independent of $\Y(0)$. If the Poisson clock rings at site $k$ at time $t$, the value at site $k$ is updated to $Y_k(t) = \sum_{j=1}^np_{kj}Y_j(t-)$ and the masses at the other sites remain unchanged. For $i \in \mathcal{I}$, we will write $\Y^{(i)}$ for the smoothing process with initial mass
\begin{align}\label{d25.1}
Y^{(i)}_j(0)=\delta_i(j), \qquad j \in \mathcal{I}.
\end{align}

We will write $\DD(\,\cdot\,)$ to denote the distribution of a random object, such as a random variable or stochastic process.

\begin{remark}\label{d26.5}

Let $\mathcal{P}$ denote the collection of Poisson processes (representing clock rings at sites in $\mathcal{I}$) used to construct the potlatch process $\XX$. Use the same collection of Poisson processes $\mathcal{P}$ to construct the family of smoothing processes $\{\Y^{(i)}(\,\cdot\,) : i\in \mathcal{I}\}$ with initial distributions as in \eqref{d25.1}. Let $\wt X_i(t) = \sum_{j=1}^n X_j(0)Y^{(i)}_j(t)$, $t \ge 0$,
$i\in\mathcal{I}$, and $\wt \XX(t) = (\wt X_i(t))_{i\in \mathcal{I}}$.
From the dual representation (see \cite[Thm. 2.3]{holley1981generalized}), it follows that  for each fixed $t \ge 0$,
\begin{align}\label{d11.1}
\DD\left( \XX(t) \right) = \DD\left(\wt \XX(t) \right).
\end{align}
\end{remark}

 Let $\sgt$ denote the  spectral gap of the `two-step' Markov chain with transition matrix $P^2$ (see Section \ref{prelim} for a review of the spectral theory). 
For $\y = (y_1, y_2,\dots, y_n)$, let
\begin{align}
\by &= \sum_{i=1}^n\pi_iy_i,\notag\\
\Vy(\y) &= \sum_{i=1}^n\pi_i(y_i - \by)^2 
= \left(\sum_{i=1}^n \pi_i y_i^2 \right)- \by^2,
\label{n26.1}\\
\Vt(t) &= \Vy(\Y(t)).\label{d2.1}
\end{align}

Our first theorem gives a rate at which the global profile of the smoothing process converges to equilibrium as captured by the variance functional
$\Vt(t)$. 

\begin{theorem}\label{global}
Assume that $n<\infty$.
For any $t \ge s \ge 0$,
$$
\mathbb{E} \Vt(t) \le \mathbb{E} \Vt(s)
\exp\left(-\sgt (t-s)\right).
$$
Consequently, for any $t \ge 0$,
$$
\mathbb{E} \Vt(t) \le \mathbb{E} \Vt(0)\exp\left(-\sgt t\right).
$$
\end{theorem}

Theorem \ref{global}, along with duality, is used to obtain convergence rates to stationarity in $L^2$-Wasserstein distance for the potlatch process in Theorem \ref{meteorconv}. Let $d^{(2)}_W(\mu,\nu)$ denote the $L^2$-Wasserstein distance between the probability measures $\mu$ and $\nu$
(see \cite[Ch. 6]{Vill}). Write $\pi_* := \min _{1 \le j \le n}\pi_j$ and $\pi^* := \max _{1 \le j \le n}\pi_j$. Recall that we assumed that the Markov chain with transition matrix $P$ is irreducible. Hence, if $n<\infty$ then $\pi_* >0$.

\begin{theorem}\label{meteorconv}
Assume that $n<\infty$.
Starting from any (possibly random) initial configuration $\XX(0) = (X_i(0))_{1 \le i \le n}$, the potlatch process  $\XX(t)$ converges weakly to a random vector $\XX(\infty)$ as $t \rightarrow \infty$. Moreover, for any $t\ge 0$,
\begin{align}\label{d11.2}
d^{(2)}_W(\DD(\XX(t)),\DD(\XX(\infty))) \le \sqrt{\frac{2\pi^*}{\pi_*}n \ \mathbb{E}\left[\left(\sum_{j=1}^nX_j(0)\right)^2\right]}\exp\left(-\sgt t/2\right).
\end{align}
\end{theorem}

\begin{remark}
Clearly, the distribution of $\XX(\infty)$ is the unique (by \eqref{d11.2}) stationary distribution for the potlatch process $\XX(t)$.
The first assertion of Theorem \ref{meteorconv} was also proved in  \cite[Thm. 3.2]{BBPS}. It was shown that  the stationary distribution is unique for a fixed total initial mass $\sum_{j=1}^nX_i(0)$. However, the only result concerning rates of convergence to stationarity for these processes on general finite graphs obtained in \cite{BBPS} is Theorem 3.4 which connects the $L^1$-Wasserstein distance to stationarity to the meeting time of independent continuous time random walks on the graph. The available estimates on these meeting times on general graphs are much weaker (see \cite[Remark 3.5]{BBPS}) than the rates of convergence obtained in Theorem \ref{meteorconv} via the duality approach. Moreover, \cite{BBPS} asked whether the convergence rates to stationarity of the potlatch process considered there can be connected to the mixing time and spectral gap of the associated random walk on the graph (see the Conjecture and discussion preceding \cite[Prop. 3.8]{BBPS}). Theorem \ref{meteorconv} provides the first result in this direction. Convergence rates for the potlatch process with $P$ corresponding to random walk on the torus, which were obtained in detail in \cite[Thm. 3.6]{BBPS}, are improved in Section \ref{torus} below for certain initial mass distributions. 
\end{remark}

Next, we present a detailed analysis of the local and global smoothing properties of the smoothing process on the torus $\T_n^d := (\mathbb{Z}/n\mathbb{Z})^d$ for odd $n$ with $P$ taken as the transition matrix of the simple random walk on this graph. 
We define the
energy functional by
\begin{align}\label{d19.8}
\Et(t) & = \frac{1}{4dn^d}\sum_{\ii \sim \jj}(Y_{\ii}(t) - Y_{\jj}(t))^2,
\end{align}
where  $\sum_{\ii \sim \jj}$ denotes sum over all ordered pairs $(\ii,\jj)$, $\ii, \jj \in \T_n^d$ such that $\ii$ and $\jj$ are neighboring vertices on the torus. The following proposition shows that, almost surely, the energy functional is non-increasing in time and thus, the smoothing process indeed ``smooths out" the mass profile in time.
\begin{proposition}\label{d19.10}
Almost surely, $t \mapsto \EE(t)$ is non-increasing in time $t$.
\end{proposition}

In the context of the averaging process, \cite[Prop. 4]{aldous2012lecture} contains an implicit bound on the expectation of the energy functional which heuristically suggested that the mass profile becomes locally smooth at a faster rate than the rate of global smoothing (we obtain an analogous bound for the smoothing process with general reversible kernels in Corollary \ref{local}). An open problem proposed in \cite[Sec. 3.3 (i)]{aldous2012lecture} was to obtain explicit bounds on the energy functional and mathematically justifying this heuristic. We provide the first mathematical justification for this heuristic in the context of the smoothing process on the torus in Theorem \ref{wass} below when started from unit mass at the origin and zero mass elsewhere. The theorem quantifies the rates of global and local smoothing of the mass profile in terms of the variance functional defined in \eqref{d2.1} and energy functional  defined in \eqref{d19.8}. The difference in orders of the polynomial decay term for the expected energy and variance functionals obtained in Theorem \ref{wass} shows that local smoothing indeed happens faster than global smoothing.

Let $\sg_1 := 1 - \cos(2\pi /n)$ denote the spectral gap of the simple random walk on $\T_n^1$ (see Section \ref{prelim} for a discussion of spectral properties of this process).

\begin{theorem}\label{wass}
Consider the smoothing process $\YY = \YY^{(\zero)}$ on $\T_n^d$. There exist constants $\alpha_1,\alpha_2,\alpha_3 >0$, $n_0 \ge 3, t_0 >0$ (all depending on $d$) such that for all $n \ge n_0$, $t \ge t_0$,
\begin{align}\label{d12.1}
\frac{\alpha_1 n^{-d}}{t^{\frac{d}{2} + 1} \wedge n^{d+2}} 
\exp\left(-2\sg_1 d^{-1} t\right) 
&\le \mathbb{E}\Et(t)\\
&\le \frac{\alpha_2 n^{-d}}{t^{\frac{d}{2} + 1} \wedge n^{d+2}} 
\exp\left(-\left(2\sg_1 d^{-1} - \frac{\alpha_3}{n^{d+4}}\right)  t\right),\notag
\end{align}
\begin{align}\label{d12.2}
\frac{\alpha_1 n^{-d}}{t^{\frac{d}{2}} \wedge n^{d}} 
\exp\left(-\left(2\sg_1 d^{-1} + \frac{\alpha_3}{n^{d+4}}\right) t\right) 
&\le \mathbb{E}\Vt(t)\\
&\le \frac{\alpha_2 n^{-d}}{t^{\frac{d}{2}} \wedge n^{d}} 
\exp\left(-\left(2\sg_1 d^{-1} - \frac{\alpha_3}{n^{d+4}}\right) t\right).\notag
\end{align}
\end{theorem}

\begin{remark}\label{expopt}
(i)
The exponents in the upper and lower bounds  in \eqref{d12.1}-\eqref{d12.2} are different. 
However, the ratio of the upper and lower bounds  is bounded above by a constant depending only on $d$ as long as $t\leq n^{d+4}$ (note that for large $n$, $n^{d+4}$ is much larger than the relaxation time for the simple random walk on $\T^d_n$ which is of order $n^2$; see Section \ref{prelim}).
 
 (ii) For the smoothing process associated to simple random walk on $\T^d_n$, apart from capturing the polynomial decay terms that dictate convergence rates for $t \le n^2$, Theorem \ref{wass} also furnishes improved exponential terms in the bounds, relative to Theorem \ref{global}. To see this, note that if one directly applies Theorem \ref{global} to the smoothing process (associated to simple random walk) on $\T^d_n$, one obtains an upper bound on the expected variance which decays like $\exp\left(-\sgt_d t\right)$, where $\sgt_d$ is the spectral gap of the two-step simple random walk on $\T^d_n$ (see Section \ref{prelim}). Remark \ref{d12.3} shows that for $d \le 3$, the quantities $2\sg_1/d$, $\sgt_d$ and $2\sg_1/d - \sgt_d$ are all positive and  $O(n^{-2})$ for large $n$, whereas for $d \ge 4$, $2\sg_1/d$ and $\sgt_d$ are  $O(n^{-2})$ while their difference is $O(n^{-4})$ for large $n$. Thus, the exponential decay term in the upper bound in \eqref{d12.2} decays faster than the bound in Theorem \ref{global} (for large $n$) for $t > n^2$ in the case $d \le 3$ and for $t > n^4$ in the case $d \ge 4$.
\end{remark}

Theorem \ref{wass} can be used to obtain an upper bound on the rate at which the average mass of all the sites converges to equilibrium starting from an arbitrary initial configuration. Let
\begin{align}\label{d25.2}
\overline{Y}(t) = n^{-d}\sum_{\ii \in \T_n^d}Y_\ii(t).
\end{align}
It is elementary to prove  that for the smoothing process on a finite state space, the limit 
$\overline{Y}(\infty) := \lim_{t \rightarrow \infty} \overline{Y}(t)$  exists almost surely and, moreover, the mass at each site converges almost surely to $\overline{Y}(\infty)$.

\begin{theorem}\label{avgquick}
Suppose that $\{\YY(t) : t \ge 0\}$ is a smoothing process with  (possibly random) starting configuration $Y_{\ii}(0) = y_{\ii} \in \mathbb{R}$, $\ii \in \T_n^d$. Denote the median of the values $\{y_{\ii}: \ii \in \T_n^d\}$ by $y^*$. 
Recall the constants $\alpha_2, \alpha_3$ from Theorem \ref{wass}.
There exist  constants $ n_0 \ge 3$ and $ t_0 >0$ (depending on $d$) such that for all $n \ge n_0$, $t \ge t_0$,
\begin{align}\label{d8.1}
&\mathbb{E}\left(\overline{Y}(t) - \overline{Y}(\infty)\right)^2\\
 &\le
\frac{\alpha_2}{t^{\frac{d}{2} + 1} \wedge n^{d + 2}} 
\mathbb{E}
\left[\left(n^{-d}\sum_{\ii \in \T_n^d}|y_{\ii} - y^*|\right)^2\right] 
\exp\left(-\left(2\sg_1 d^{-1} - \frac{\alpha_3}{n^{d+4}}\right) t\right)\nonumber.
\end{align}
\end{theorem}

\begin{remark}\label{avgrem}
The bound in \eqref{d8.1} is essentially sharp when the initial mass is non-zero at the origin and zero elsewhere, in the sense that we obtain an `almost' matching lower bound  (see Theorem \ref{locsmooth} (i) below). In particular, with such a starting configuration, \eqref{d12.2} and \eqref{d8.1} show that for large $n$, the quantity $\mathbb{E}\left(\overline{Y}(t) - \overline{Y}(\infty)\right)^2$ reaches a fraction of its initial value faster than the time taken by $\mathbb{E}\Vt(t)$ to reach the same fraction. In the context of opinion dynamics, this can be restated as ``the average opinion stabilizes faster than the global opinion.''
\end{remark}

Using duality, we obtain in Theorem \ref{metone} an upper bound on the rate of convergence of the potlatch process to stationarity in $L^2$-Wasserstein distance for a class of starting configurations. It improves the bound in  \cite[Thm. 3.6]{BBPS} for these starting configurations and obtains the polynomial correction to the exponential convergence rate which essentially governs convergence to stationarity till time $n^2$ before the exponential convergence rate sets in.
\begin{theorem}\label{metone}
Recall the constants $\alpha_2, \alpha_3$ from Theorem \ref{wass}.

(i) Consider the potlatch process $\{\XX(t): t \ge 0\}$ on $\T_n^d$ with starting configuration $X_{\jj}(0) = 1$ for all $\jj \in \T_n^d$. Then there exists $n_0 \ge 3$ such that for any $n \ge n_0$, $\ii \in \T_n^d$ and $t \ge 0$,
\begin{align}\label{d26.1}
d^{(2)}_W(\DD(X_{\ii}(t)),\DD(X_{\ii}(\infty))) \le 
\sqrt{\frac{\alpha_2}{t^{\frac{d}{2} + 1} \wedge n^{d+2}}}\exp\left(-\left(\sg_1 d^{-1} - \frac{\alpha_3}{2n^{d+4}}\right) t\right).
\end{align}
Moreover, 
\begin{align}\label{d26.2}
d^{(2)}_W(\DD(\XX(t)),\DD(\XX(\infty))) \le 
\sqrt{\frac{\alpha_2n^d}{t^{\frac{d}{2} + 1} \wedge n^{d+2}}}\exp\left(-\left(\sg_1 d^{-1} - \frac{\alpha_3}{2n^{d+4}}\right) t\right).
\end{align}
(ii) Consider the potlatch process $\{\XX(t): t \ge 0\}$ on $\T_n^d$ with starting configuration $X_{\jj}(0) = X_{\jj}$ for all $\jj \in \T_n^d$, where $\{X_{\jj} : \jj \in \mathbb{Z}^d\}$ are non-negative random variables for which $\operatorname{Cov}\left(X_{\jj}, X_{\kk}\right) = 0$ for $\jj \neq \kk$, $\mathbb{E}(X_{\jj}) = \mu < \infty$ for all $\jj \in \mathbb{Z}^d$ and $\zeta := \sup_{\jj \in \mathbb{Z}^d}\mathbb{E}(X_{\jj}^2) < \infty$. Then there exist $\eta>0$ and $n_0 \ge 3$ (depending on $\mu$, $\zeta$ and $d$)  such that for any $n \ge n_0$, $\ii \in \T_n^d$ and $t \ge 0$,
\begin{align}\label{d26.3}
d^{(2)}_W(\DD(X_{\ii}(t)),\DD(X_{\ii}(\infty))) \le \sqrt{\frac{\eta}{t^{\frac{d}{2}} \wedge n^{d}}}\exp\left(-\left(\sg_1 d^{-1} - \frac{\alpha_3}{2n^{d+4}}\right) t\right).
\end{align}
Moreover, 
\begin{align}\label{d26.4}
d^{(2)}_W(\DD(\XX(t)),\DD(\XX(\infty))) \le \sqrt{\frac{\eta n^d}{t^{\frac{d}{2}} \wedge n^{d}}}\exp\left(-\left(\sg_1 d^{-1} - \frac{\alpha_3}{2n^{d+4}}\right) t\right).
\end{align}
\end{theorem}

By taking a limit as $n\rightarrow \infty$ in an appropriate way, Theorem \ref{metone} can be used to obtain rates of convergence to local equilibrium for the potlatch process on $\mathbb{Z}^d$. 

For $j \ge 3$ odd, let
\begin{align}\label{d27.1}
\mathbb{B}^d_j =\{(k_1,\dots,k_d) \in \mathbb{Z}^d: -(j-1)/2 \le k_i \le (j-1)/2 \text{ for } 1 \le i \le d\}.
\end{align}
We will call a function $\phi: \mathbb{R}^{\mathbb{Z}^d} \rightarrow \mathbb{R}$ Lipschitz and of bounded support if there exists odd $j \ge 3$ such that $\phi$ depends only on the coordinates in $\mathbb{B}^d_j $  and there exists $\rho>0$ such that for any $\xi_1, \xi_2 \in \mathbb{R}^{\mathbb{Z}^d}$, 
\begin{align}\label{d28.8}
|\phi(\xi_1) - \phi(\xi_2)| \le \rho\left(\sum_{\ii \in \mathbb{B}_j^d}\left(\xi_1(\ii) - \xi_2(\ii)\right)^2\right)^{1/2}.
\end{align}
 
\begin{theorem}\label{meteorzd}
Consider the potlatch process $\{\XX(t): t \ge 0\}$ on $\mathbb{Z}^d$ with starting configuration $X_{\jj}(0) = X_{\jj}$ for all $\jj \in \mathbb{Z}^d$, where $\{X_{\jj} : \jj \in \mathbb{Z}^d\}$ are non-negative random variables satisfying $\operatorname{Cov}\left(X_{\jj}, X_{\kk}\right) = 0$ for $\jj \neq \kk$, $\mathbb{E}(X_{\jj}) = \mu < \infty$ for all $\jj$, and $\sup_{\jj \in \mathbb{Z}^d}\mathbb{E}(X_{\jj}^2) < \infty$. Then $\XX(t)$ converges weakly to a stationary random mass distribution $\XX(\infty)$ as $t \rightarrow \infty$. Moreover, for any function $\phi: \mathbb{R}^{\mathbb{Z}^d} \rightarrow \mathbb{R}$ which is Lipschitz and of bounded support, there exists a positive constant $C_{\phi}$ such that
$$
d^{(2)}_W(\DD(\phi(\XX(t))),\DD(\phi(\XX(\infty)))) \le C_{\phi}t^{-d/4}, \quad t \ge 1.
$$
The bound can be improved to $C_{\phi}t^{-(d+2)/4}$ if $X_{\jj}(0) = 1$ for all $\jj \in \mathbb{Z}^d$.
\end{theorem}

\section{Preliminaries}\label{prelim}

This section is devoted to a review of some results from the spectral theory for Markov chains.

Consider a Markov chain with a finite state space $\mathcal{I}$ and transition matrix $P$. Recall that we assume that the Markov chain is irreducible, aperiodic and  reversible with respect to $\pi$. 
Let 
\begin{align*}
1 = \lambda_0 > \lambda _1 \geq \dots \geq \lambda_{n-1} > -1
\end{align*}
be the eigenvalues of $P$ and note that  $\lambda_{n-1} > -1$ because we have assumed that the Markov chain is aperiodic (see  \cite[Lemma 12.1]{Levin2009}). 
Let $\sga$ denote the absolute spectral gap of this Markov chain, i.e., 
$\sga = 1- \lambda_*$, 
where $\lambda_* =\max\{|\lambda_i| : \lambda_i < 1\}$ (see  \cite[Sect. 12.2]{Levin2009}).
Note that $\sga>0$ since $\lambda_{n-1} > -1$. 

Since $\{\lambda_i\}_{0 \le i \le n-1}$ are the eigenvalues of $P$,  $\{\lambda_i^2\}_{0 \le i \le n-1}$ are the eigenvalues of $P^2$. This follows, for example, from the spectral representation \cite[(12.2)]{Levin2009}.
The (absolute) spectral gap $\sgt$ of the ``two-step'' Markov chain is equal to
$ 1- \lambda^2_*$ because $\lambda^2_* =\max\{|\lambda^2_i| : \lambda^2_i < 1\}$. Hence
\begin{align}\label{d22.7}
\sgt = 1- \lambda^2_* = 1 - (1- \sga)^2 = 2\sga - (\sga)^2.
\end{align}

Consider the simple random walk on $\T_n^1$ and assume that $n$ is odd. Then the process is aperiodic and its eigenvalues are $\cos(2\pi k/n)$, $k=0,1,\dots, n-1$ (see \cite[Sec. 12.3.1]{Levin2009}).

The simple random walk on $\T_n^d$ can be expressed as a product chain of $d$ simple random walk chains on $\T_n^1$ (see  \cite[Sec. 12.4]{Levin2009}).
We will write $\sg_d$ for the spectral gap of the simple random walk on $\T_n^d$ and we will denote the spectral gap of the ``two-step'' random walk on $\T_n^d$ by $\sgt_d$. We have 
\begin{align}\label{d12.4}
\sg_1 = 1 - \cos(2\pi /n).
\end{align}

 By  \cite[Lemma 12.11]{Levin2009}, the $n^d$ eigenvalues of the transition matrix for the product chain (counting multiplicities) are given by the set $\mathcal{S} =\{d^{-1}\sum_{j=1}^{d}\cos(2\pi k_j/n) : 0 \le k_1, \dots, k_d \le n-1\}$. 
It follows that $\sg_d = \sg_1/d$.
 
 As $n \ge 3$ is odd, $\min\{\cos(2\pi k/n) : 0 \le k \le n-1\} = -\cos(\pi/n)$. Using this,
$$
\max\{|\lambda| : \lambda \in \mathcal{S}, \lambda \neq 1\} = \max\left\lbrace\frac{d-1 + \cos(2\pi /n)}{d}, \cos(\pi/n)\right\rbrace.
$$
Thus,  in view of \eqref{d22.7},
\begin{align*}
\sgt_d &= 1 - \left(\max\left\lbrace\frac{d-1 + \cos(2\pi /n)}{d}, \cos(\pi/n)\right\rbrace\right)^2 \\
&= \min\left\lbrace\frac{2\sg_1}{d} - \frac{(\sg_1)^2}{d^2}, 1 - \cos^2(\pi/n)\right\rbrace.
\end{align*}
The following  can be easily checked, 
\begin{align*}
\frac{d-1 + \cos(2\pi /n)}{d} &\leq \cos(\pi/n) \qquad \text{if  }  d\leq 3, n\geq 3, n  \text{ odd},\\
\frac{d-1 + \cos(2\pi /n)}{d} &> \cos(\pi/n) \qquad \text{if  } d\geq 4, n\geq 3, n  \text{ odd}.
\end{align*}
Hence,
\begin{align}\label{d12.5}
\sgt _d
&= 1 - \cos^2(\pi/n) \qquad \text{if  }  d\leq 3, n\geq 3,  n  \text{ odd},\\
\sgt_d 
&= \frac{2\sg_1}{d} - \frac{(\sg_1)^2}{d^2} \qquad \text{if  }  d\geq 4, n\geq 3,  n  \text{ odd}.\label{d12.6}
\end{align}

\begin{remark}\label{d12.3}
It follows from \eqref{d12.4}-\eqref{d12.6} that for $d\leq 3$,
\begin{align*}
0 &\leq \frac{2\sg_1}{d} -\sgt_d = 2(1 - \cos(2\pi/n))/d -( 1 - \cos^2(\pi/n)) = \frac{4-d}{d} \sin ^2\left(\pi /n\right).
\end{align*}
We see that $2\sg_1/d$, $\sgt_d$ and their difference are all $O(n^{-2})$, for $d\leq 3$.

For $d\geq 4$,
\begin{align}\label{d10.1}
0 \leq \frac{2\sg_1}{d} -\sgt_d = \frac{(\sg_1)^2}{d^2}.
\end{align}
Both $2\sg_1/d$ and $\sgt_d$ are $O(n^{-2})$ while $(\sg_1)^2/d^2$ is $O(n^{-4})$ for large $n$. 
\end{remark}

\begin{remark}\label{d19.4}
Next we present the spectral representation of the  heat kernel on $\T_n^d$.
We will sometimes identify $\T_n^1$ with $\{0,1, \dots, n-1\}$ in our notation.

Let $p: \mathbb{R}_+ \times \T_n^d \rightarrow \mathbb{R}$ be the heat kernel of the continuous time simple random walk on $\T_n^d$ starting from $\zero$. 

Let
$\hl_{\x}^{(d)} = d^{-1}\sum_{j=1}^d(1-\cos(2\pi x_j/n))$ for $ \x =(x_1, x_2, \dots,x_n) \in \T_n^d$. For $d=1$ we will write $\hl_j = \hl^{(1)}_\x$, where $\x=\{j\}$. In other words, $\hl_j = 1-\lambda_j$ (see \cite[Lemma 1.3.3]{SC} for the relationship between the eigenvalues of the discrete time Markov chain and its continuous time version). By \eqref{d12.4}, $\hl_1 = \sg_1$.

Using the bounds  
\begin{align}\label{d16.1}
\frac{x^2}{2} - \frac{x^4}{4!} \le 1 - \cos (x) \le \frac{x^2}{2},
\qquad x \geq 0,
\end{align}
one obtains constants $C_1, C_2>0$ such that
\begin{align}\label{d11.5}
C_1\frac{j^2}{n^2} \le \hl_j \le C_2\frac{j^2}{n^2}, \qquad 1 \le j \le \lfloor n/2\rfloor, \ n\geq 3.
\end{align}
Hence,
\begin{align}\label{eigbd}
0 <
\inf_{n \ge 3}n^2\hl_1  
=\inf_{n \ge 3}n^2\sg_1  
\leq  \sup_{n \ge 3}n^2\hl_1  
= \sup_{n \ge 3}n^2\sg_1   < \infty.
\end{align}


Let $\psi_\x(\,\cdot\,)$ be an eigenfunction corresponding to $\hl^{(d)}_\x$ and assume that $\{\psi_{\x}(\,\cdot\,), \x \in \T_n^d\}$ form an orthonormal basis for $L^2(\T_n^d)$ relative to the counting measure.
In other words, for $\x\ne\y$,
\begin{align}\label{d19.5}
\sum_{\ii \in \T_n^d} \psi_{\x}(\ii)^2 =1, \qquad
\sum_{\ii \in \T_n^d} \psi_{\x}(\ii)\psi_{\y}(\ii) =0.
\end{align}
By \cite[Lemma 1.3.3]{SC},
\begin{equation}\label{d19.1}
p(t,\ii) = \frac{1}{\sqrt{n^{d}}}\sum_{\x \in \T_n^d}e^{-\hl_{\x}^{(d)}t} \psi_{\x}(\ii), \qquad t\geq 0, \ii \in \T_n^d.
\end{equation}

For $d=1$, we have the following form of the spectral representation,
\begin{equation}\label{d19.2}
p(t,\ii) = \frac{1}{\sqrt{n}}\psi_0(\ii) + \sqrt{\frac{2}{n}}\sum_{j=1}^{\lfloor n/2\rfloor}e^{-\hl_jt}\psi_j(\ii), \qquad t\geq 0, \ii \in \T_n^1,
\end{equation}
where $\psi_0(\ii) = 1/\sqrt{n}$, $\hl_j = 1-\lambda_j = 1 - \cos(2\pi j/n)$,
$\psi_j(\ii) = \sqrt{2/n}\cos(2\pi j\ii/n)$, for $ \ii \in \T_n^1$ and $1 \le j \le \lfloor n/2\rfloor$. 
\end{remark}

\section{Convergence rates for general reversible kernels}\label{revker}
In this section we prove Theorem \ref{global} and Theorem \ref{meteorconv} using spectral theory for reversible Markov chains. 
We adapt the methods in \cite{aldous2012lecture}, which addresses convergence rates for the averaging process, to the smoothing process. However, note that unlike the averaging process, the smoothing process does not conserve mass. 

We note here that the generator of the smoothing process, which we write as $\mathcal{L}$, acts on any function $f: \mathbb{R}^n \rightarrow \mathbb{R}$ by
\begin{align}\label{d19.7}
\mathcal{L} f(\y) = \sum_{i = 1}^n \left[f(\mathbf{p}^{(i)}\y) - f(\y)\right],
\end{align}
where 
\[(\mathbf{p}^{(i)}\y)_k  = \left\{ \begin{array}{ll}
		y_k & \mbox{ if } k \neq i,\\ 
		\sum_{j=1}^np_{ij}y_j      & \mbox{ if } k = i.
	\end{array} \right.\]

\begin{proof}[Proof of Theorem \ref{global}]
Recall \eqref{n26.1} and write $\Vy(\y) = \Vy_1(\y) - \Vy_2(\y)$,
where $\Vy_1(\y) = \sum_{i=1}^n \pi_i y_i^2$ and $\Vy_2(\y) = \by^2$.  Note that, using reversibility of $P$ with respect to $\pi$,
\begin{align}\label{lyap1}
\LL \Vy_1(\y) 
&=  \sum_{i = 1}^n \left[\Vy_1(\mathbf{p}^{(i)}\y) - \Vy_1(\y)\right]\\
&= \sum_{i = 1}^n \left[
\sum_{j\ne i} \pi_j y_j^2
+  \pi_i \left(\sum_{j=1}^np_{ij}y_j\right)^2
 - \sum_{j=1}^n \pi_j y_j^2\right]\notag\\
&= \sum_{i=1}^n \pi_i\left(\left(\sum_{j=1}^np_{ij}y_j\right)^2 - y_i^2\right) = \sum_{i=1}^n\sum_{j,k=1}^n\pi_i p_{ij}p_{ik}y_jy_k - \sum_{i=1}^n\pi_iy_i^2\notag\\
&= \sum_{j,k=1}^n\sum_{i=1}^n\pi_j p_{ji}p_{ik}y_jy_k - \sum_{i=1}^n\pi_iy_i^2 = \sum_{j,k=1}^n\pi_j p^{(2)}_{jk}y_jy_k - \sum_{i=1}^n\pi_iy_i^2,\notag
\end{align}
where $p^{(2)}_{jk}$ is the transition probability from $j$ to $k$ of the two-step Markov chain with transition matrix $P^2$. Reversibility  implies $\pi_jp^{(2)}_{jk} = \pi_kp^{(2)}_{kj}$ for any $1 \le j,k \le n$. Thus, 
$$
\sum_{i=1}^n\pi_iy_i^2 = \sum_{i,j=1}^n \pi_i p^{(2)}_{ij}y_i^2 = \sum_{i,j=1}^n\pi_j p^{(2)}_{ji}y_i^2
$$
which, in turn, gives us
\begin{equation}\label{lyap2}
\sum_{i=1}^n\pi_iy_i^2 = \frac{1}{2}\left(\sum_{j,k=1}^n \pi_j p^{(2)}_{jk}y_j^2 + \sum_{j,k=1}^n \pi_j p^{(2)}_{jk}y_k^2\right).
\end{equation}
Using \eqref{lyap2} in \eqref{lyap1}, we obtain
\begin{align}\label{lyap3}
\LL \Vy_1(\y) &= \sum_{j,k=1}^n\pi_j p^{(2)}_{jk}y_jy_k - \frac{1}{2}\left(\sum_{j,k=1}^n \pi_j p^{(2)}_{jk}y_j^2 + \sum_{j,k=1}^n \pi_j p^{(2)}_{jk}y_k^2\right)\\
& = -\frac{1}{2}\sum_{j,k=1}^n\pi_jp^{(2)}_{jk}\left(y_j-y_k\right)^2. \notag
\end{align}
Write $Q = (q_{ij})_{1 \le i,j\le n}$ for the  matrix  $Q := P - I$. We obtain
\begin{align}\label{lyap4}
\LL \Vy_2(\y)
&=  \sum_{i = 1}^n \left[\Vy_2(\mathbf{p}^{(i)}\y) - \Vy_2(\y)\right] \\
&= \sum_{i=1}^n\left(\left(\sum_{j\neq i}\pi_jy_j + \pi_i\sum_{k=1}^np_{ik}y_k\right)^2 - \left(\sum_{j=1}^n\pi_jy_j\right)^2\right)\notag\\
 &= \sum_{i=1}^n\left(\left(\sum_{j=1}^n\pi_jy_j + \pi_i\sum_{k=1}^nq_{ik}y_k\right)^2 - \left(\sum_{j=1}^n\pi_jy_j\right)^2\right) \notag \\
  &= \sum_{i=1}^n\left(2\left(\sum_{j=1}^n\pi_jy_j\right)\left(\pi_i\sum_{k=1}^nq_{ik}y_k\right) + \left(\pi_i\sum_{k=1}^nq_{ik}y_k\right)^2\right) \notag \\
  & = \sum_{i=1}^n\left(\pi_i\sum_{k=1}^nq_{ik}y_k\right)^2,\notag
\end{align}
where the last equality is a consequence of the fact that
\begin{align*}
\sum_{i=1}^n\pi_i\sum_{k=1}^nq_{ik}y_k &= \sum_{k=1}^n\sum_{i=1}^n\pi_iq_{ik}y_k = \sum_{k=1}^n\sum_{i=1}^n\pi_i(p_{ik} - \delta_{ik})y_k \\
&= \sum_{k=1}^n\left(\sum_{i=1}^n\pi_kp_{ki}y_k - \pi_ky_k\right) = 0.
\end{align*}
Combining \eqref{lyap3} and \eqref{lyap4}, we obtain
\begin{equation}\label{lyap5}
\LL \Vy(\y) = \LL \Vy_1(\y) - \LL \Vy_2(\y) = -\frac{1}{2}\sum_{j,k=1}^n\pi_jp^{(2)}_{jk}\left(y_j-y_k\right)^2 - \sum_{i=1}^n\left(\pi_i\sum_{k=1}^nq_{ik}y_k\right)^2.
\end{equation}
Define the `two-step' energy functional 
\begin{align}\label{d10.3}
\Ey^{(2)}(\y) = \frac{1}{2}\sum_{j,k=1}^n\pi_jp^{(2)}_{jk}\left(y_j-y_k\right)^2.
\end{align}
The variational representation for the spectral gap $\sgt$ of the `two-step' Markov chain on $\mathcal{I}$ with transition matrix $P^2$ (\cite[Remark 13.13]{Levin2009}) implies that
$$
\sgt \leq \inf_{\y \neq 0} \frac{\Ey^{(2)}(\y)}{\Vy(\y)}.
$$
Using this in \eqref{lyap5},
\begin{equation}\label{locref}
\LL \Vy(\y) \le -\Ey^{(2)}(\y) \le -\sgt \Vy(\y).
\end{equation}
Hence,
$$
\frac{d}{dt}\mathbb{E} \Vt(t) \le -\sgt \mathbb{E} \Vt(t),
$$
which gives for any $t \ge s \ge 0$,
$$
\mathbb{E} \Vt(t) \le e^{-\sgt (t-s)}\mathbb{E} \Vt(s).
$$
\end{proof}

Corollary \ref{local} given below shows that
the proof of Theorem \ref{global} yields an implicit bound on the rate at which the mass profile becomes ``locally smooth.'' This is expressed in terms of the `two-step' energy functional \eqref{d10.3}. 
Let
\begin{align*}
\Ey^*(\y) &= \sum_{i=1}^n\left(\pi_i\sum_{k=1}^nq_{ik}y_k\right)^2,\qquad
\Et^{(2)}(t) = \Ey^{(2)}(\Y(t)), \qquad 
\Et^{*}(t) = \Ey^{*}(\Y(t)).
\end{align*}

\begin{corollary}\label{local}
We have
$$
\int_0^{\infty}\mathbb{E}\left(\Et^{(2)}(t) + \Et^{*}(t)\right)dt = \mathbb{E} \Vt(0).
$$
\end{corollary}
\begin{proof}
From \eqref{lyap5},
$$
\frac{d}{dt}\mathbb{E} \Vt(t) = - \mathbb{E}\left(\Et^{(2)}(t) + \Et^{*}(t)\right).
$$
This, along with the observation $\lim_{t \rightarrow \infty} \mathbb{E} \Vt(t) = 0$, which immediately follows from Theorem \ref{global}, implies the result.
\end{proof}
\begin{remark}
The above bound suggests at a heuristic level that local smoothing happens at a faster rate than global smoothing. This was also observed in \cite[Sect. 2.3]{aldous2012lecture}. Theorem \ref{wass} verifies this heuristic for the smoothing process on the torus.
\end{remark}
Using Theorem \ref{global} and duality, we now prove Theorem \ref{meteorconv}.
\begin{proof}[Proof of Theorem \ref{meteorconv}]
Recall the process $\Y^{(i)}$ defined in \eqref{d25.1}. Writing $M^{(i)}(t) = \max_{1 \le j \le n}Y^{(i)}_j(t)$ and $m^{(i)}(t) = \min_{1 \le j \le n}Y^{(i)}_j(t)$, note that the dynamics of the smoothing process implies that $M^{(i)}(t)$ is non-increasing in time and $m^{(i)}(t)$ is non-decreasing in time. Recall the function $\Vy$ from  \eqref{n26.1} and let $\bY^{(i)}(t) = \sum_{j=1}^n\pi_jY^{(i)}_j(t)$, $t \ge 0$. 
Note that $\overline{Y}^{(i)}(0) = \pi_i$ and, therefore,
$$
\Vy(\Y^{(i)}(0)) = \pi_i(1-\pi_i)^2 + \sum_{j \neq i} \pi_j(-\pi_i)^2 = \pi_i(1-\pi_i)^2 + (1-\pi_i)\pi_i^2 = \pi_i(1-\pi_i)\le \pi^*.
$$
We apply this formula and Theorem \ref{global} to see that,
\begin{align}\label{mc1}
\mathbb{E}&\left[\left(M^{(i)}(t) - m^{(i)}(t)\right)^2 \right]\\
&\le 2\left(\mathbb{E}\left[\left(M^{(i)}(t) - \bY^{(i)}(t)\right)^2\right] + \mathbb{E}\left[\left(m^{(i)}(t) - \bY^{(i)}(t)\right)^2\right]\right)\notag\\
&\le 2\mathbb{E}\sum_{j=1}^n\left(Y^{(i)}_j(t)- \bY^{(i)}(t)\right)^2 
\le 2\pi_{*}^{-1}\mathbb{E}\sum_{j=1}^n\pi_{j}\left(Y^{(i)}_j(t)- \bY^{(i)}(t)\right)^2\notag\\
 &= 2\pi_{*}^{-1}\mathbb{E}\Vy(\Y^{(i)}(t))
 \le 2\pi_{*}^{-1}e^{-\sgt t}\mathbb{E}\Vy(\Y^{(i)}(0)) \le 2\frac{\pi^*}{\pi_*}e^{-\sgt t}.\notag
\end{align}

Thus, as $t \rightarrow \infty$, $\Y^{(i)}(t)$ converges almost surely to the vector $Y^{(i)}(\infty)\mathbbm{1}$ for some non-negative random variable $Y^{(i)}(\infty)$. Recall $\{\Y^{(i)}(\,\cdot\,) : 1 \le i \le n\}$ and $\wt \XX(t)$ from Remark \ref{d26.5}.
Let
$\wt X_i(\infty)=Y^{(i)}(\infty)\sum_{j=1}^n X_j(0)$ for $1 \le i \le n$. From the almost sure convergence of $\Y^{(i)}(\,\cdot\,)$ we obtain $\wt X_i(t) \rightarrow \wt X_i(\infty)$ as $t \rightarrow \infty$,
for  $1 \le i \le n$, a.s. By duality \eqref{d11.1}, for any $t \ge 0$, $\wt \XX(t)$ has the same law as $\XX(t)$. This implies the weak convergence of $\XX(t)$ to $\XX(\infty):=\wt \XX(\infty):= (\wt X_i(\infty))_{1 \le i \le n}$ proving the first claim in the theorem.
 Using the observation that for each $i$, $m^{(i)}(t) \le Y^{(i)}(\infty) \le M^{(i)}(t)$ for all $t \ge 0$, we obtain
\begin{align}\label{d25.3}
&\left[d^{(2)}_W(\DD(\XX(t)),\DD(\XX(\infty)))\right]^2
= \left[d^{(2)}_W(\DD(\wt\XX(t)),\DD(\wt\XX(\infty)))\right]^2\\
& \le \mathbb{E}\sum_{i=1}^{n}\left| \sum_{j=1}^n X_j(0)Y^{(i)}_j(t) - Y^{(i)}(\infty)\sum_{j=1}^n X_j(0)\right|^2\notag\\
& \le \mathbb{E}\sum_{i=1}^{n}\left(\sum_{j=1}^nX_j(0)\left|Y^{(i)}_j(t) - Y^{(i)}(\infty)\right| \right)^2\notag\\
& \le \mathbb{E}\sum_{i=1}^{n}\left(\sum_{j=1}^nX_j(0)\left|M^{(i)}(t) - m^{(i)}(t)\right|\right)^2\notag\\
&= \mathbb{E}\left[\left(\sum_{j=1}^nX_j(0)\right)^2\right]
\mathbb{E}\left(\sum_{i=1}^{n}\left(M^{(i)}(t) - m^{(i)}(t)\right)^2\right)
\notag\\
&  \le \frac{2\pi^*}{\pi_*}n\  \mathbb{E}\left[\left(\sum_{j=1}^nX_j(0)\right)^2 \right] \exp\left(-\sgt t\right),\notag
\end{align}
where we have used the independence between  $\{X_j(0)\}_{1 \le j \le n}$ and the smoothing process $\Y^{(i)}$ for the last equality, and \eqref{mc1} in the last inequality. 
\end{proof}

\section{Convergence rates for the smoothing and potlatch processes on the torus and lattice}\label{torus}
In this section, we investigate the smoothing process on the torus $\T_n^d := (\mathbb{Z}/n\mathbb{Z})^d$ for odd $n$ with $P$ taken as the transition matrix of the simple random walk on this graph. 
In the following, $\jj \sim \ii$ will indicate that $\ii$ and $\jj$ are neighbors in $\T_n^d$. 

\begin{proof}[Proof of Proposition \ref{d19.10}]
Recall that, almost surely, at most one Poisson clock  rings at any time. Suppose the clock at the $\ii$-th vertex rings at (random) time $t$. Then
\begin{align*}
&\EE(t) - \EE(t-) = \frac{1}{4dn^d}\sum_{\jj :\jj \sim \ii}\left[\left(\frac{1}{2d}\sum_{\kk : \kk \sim \ii} Y_{\kk}(t-) - Y_{\jj}(t-)\right)^2 - (Y_{\ii}(t-) - Y_{\jj}(t-))^2 \right]\\
&= \frac{1}{4dn^d}\sum_{\jj :\jj \sim \ii}\left[\left(\frac{1}{2d}\sum_{\kk : \kk \sim \ii} (Y_{\kk}(t-) - Y_{\ii}(t-)) + (Y_{\ii}(t-) - Y_{\jj}(t-))\right)^2\right.\\
& \qquad \qquad \qquad  \left. - (Y_{\ii}(t-) - Y_{\jj}(t-))^2 \right]\\
&=\frac{1}{2dn^d}\sum_{\jj :\jj \sim \ii}\left(\frac{1}{2d}\sum_{\kk : \kk \sim \ii} (Y_{\kk}(t-) - Y_{\ii}(t-))\right)(Y_{\ii}(t-) - Y_{\jj}(t-))\\
&\qquad \qquad \qquad + \frac{1}{4dn^d}\sum_{\jj :\jj \sim \ii}\left(\frac{1}{2d}\sum_{\kk : \kk \sim \ii} Y_{\kk}(t-) - Y_{\ii}(t-)\right)^2\\
& = -\frac{1}{n^d} \left(\frac{1}{2d}\sum_{\kk : \kk \sim \ii} Y_{\kk}(t-) - Y_{\ii}(t-)\right)^2 
+ \frac{1}{2n^d} \left(\frac{1}{2d}\sum_{\kk : \kk \sim \ii} Y_{\kk}(t-) - Y_{\ii}(t-)\right)^2\\
&= -\frac{1}{2n^d} \left(\frac{1}{2d}\sum_{\kk : \kk \sim \ii} Y_{\kk}(t-) - Y_{\ii}(t-)\right)^2 \le 0.
\end{align*}
This implies the result.
\end{proof}

Define the discrete Laplacian of $F: \mathbb{R}_+ \times \T_n^d \rightarrow \mathbb{R}$ by 
$$
(\Delta F)(t,\ii) = \frac{1}{2d}\sum_{\jj \sim \ii}F(t,\jj) - F(t,\ii), \quad  t \in \mathbb{R}_+, \ \ii \in \T_n^d.
$$
We will use the following notation, $\Delta_{\ii}(t) = (\Delta \Y)(t,\ii)$, i.e.,
\begin{align}\label{d19.9}
\Delta_{\ii}(t) =  \frac{1}{2d}\sum_{\jj \sim \ii}Y_{\jj}(t) - Y_{\ii}(t).
\end{align}
By the discrete integration by parts formula, for any function $f: \mathbb{R}_+ \times \T_n^d \rightarrow \mathbb{R}$, $u \in \mathbb{R}_+$, and $s \in [0,u]$,
\begin{align}\label{d11.4}
\sum_{\ii \in \T_n^d}f(u-s, \ii)\Delta_{\ii}(s) = \sum_{\ii \in \T_n^d}\Delta f(u-s, \ii)Y_{\ii}(s).
\end{align}

\begin{lemma}\label{heatmart}
For any solution $f: \mathbb{R}_+ \times \T_n^d \rightarrow \mathbb{R}$ to the heat equation $\partial_tf = \Delta f$ and any $u>0$, the process
$$
M_f(t) := \sum_{\ii \in \T_n^d}f(u-t, \ii)Y_{\ii}(t), \quad 0 \le t \le u,
$$
is a continuous time martingale with predictable quadratic variation
$$
\langle M \rangle_f(t) = \int_0^t \sum_{\ii \in \T_n^d}f^2(u-s, \ii)\Delta^2_{\ii}(s)ds.
$$
\end{lemma}

\begin{proof}
Consider independent Poisson processes $\{N_{\ii}(\,\cdot\,)\}_{\ii \in \T_n^d}$ on $\mathbb{R}_+$ with unit intensity. The evolution equations for the smoothing process can be expressed in terms of these processes as
$$
Y_{\ii}(t) = \int_0^t\Delta_{\ii}(s-)N_{\ii}(ds), \quad \ii \in \T_n^d.
$$
Let
$\wh{N}_{\ii}$ denote the compensated Poisson process $N_{\ii}$.
By the pathwise integration by parts formula, we obtain for any $t \in [0,u)$,
\begin{align*}
M_f(t) &= -\int_0^t\sum_{\ii \in \T_n^d}\partial_sf(u-s, \ii)Y_{\ii}(s)ds + \int_0^t\sum_{\ii \in \T_n^d}f(u-s, \ii)dY_{\ii}(s)\\
&= -\int_0^t\sum_{\ii \in \T_n^d}\partial_sf(u-s, \ii)Y_{\ii}(s)ds + \int_0^t\sum_{\ii \in \T_n^d}f(u-s, \ii)\Delta_{\ii}(s-)N_{\ii}(ds)\\
&= -\int_0^t\sum_{\ii \in \T_n^d}\partial_sf(u-s, \ii)Y_{\ii}(s)ds + \int_0^t\sum_{\ii \in \T_n^d}f(u-s, \ii)\Delta_{\ii}(s)ds\\
&\qquad\qquad + \int_0^t\sum_{\ii \in \T_n^d}f(u-s, \ii)\Delta_{\ii}(s-)\wh{N}_{\ii}(ds)\\
&= \int_0^t\sum_{\ii \in \T_n^d}\left(-\partial_sf(u-s, \ii) + \Delta f(u-s, \ii)\right)Y_{\ii}(s)ds \\
&\qquad \qquad+ \int_0^t\sum_{\ii \in \T_n^d}f(u-s, \ii)\Delta_{\ii}(s-)\wh{N}_{\ii}(ds)\\
&= \int_0^t\sum_{\ii \in \T_n^d}f(u-s, \ii)\Delta_{\ii}(s-)\wh{N}_{\ii}(ds),
\end{align*}
where the fourth equality follows 
from \eqref{d11.4} 
and the last equality follows as $f$ is a solution to the heat equation. This proves that $\{M_f(t) : 0 \le t \le u\}$ is a martingale. The formula for the predictable quadratic variation follows from the above representation and \cite[Thm. I.4.40(d), p. 48]{JSh}.
\end{proof}

For functions $f,g : \mathbb{R}_+ \rightarrow \mathbb{R}$, we denote their convolution by $f\star g$. The $k$-fold convolution of $f$ with itself is denoted by $f^{\star k}$. Further, for $\ii, \jj \in \T_n^d$, $\ii + \jj$ denotes addition modulo $n$.

Recall that $p(t,\ii)$ denotes the heat kernel of the continuous time simple random walk on $\T_n^d$ starting from $\zero$. Note that $\Delta p$ is the solution to the heat equation with initial condition $\Delta p(0,\ii) = \frac{1}{2d}\sum_{\jj \sim \ii}\delta_{\zero}(\jj) - \delta_{\zero}(\ii), \ii \in \T_n^d$. Define $G : \mathbb{R}_+ \rightarrow \mathbb{R}$ by
\begin{equation}\label{l2grad}
G(t) = \sum_{\ii \in \T_n^d} (\Delta p(t,\ii))^2, \quad t \ge 0.
\end{equation}

\begin{lemma}\label{mom2delta}
 For any $u \ge 0$,
$$
\mathbb{E}\sum_{\ii \in \T_n^d}\Delta_{\ii}^2(u) = \sum_{k=1}^{\infty}G^{\star k}(u).
$$
\end{lemma}
\begin{proof}
Fix $u \ge 0$. For any $\ii \in \T_n^d$, using Lemma \ref{heatmart} with $f_{\ii}(t,\jj) = \Delta p(t, \ii + \jj)$ (which is the solution to the heat equation with initial condition $f_{\ii}(0,\jj) = \frac{1}{2d}\sum_{\kk \sim \jj}\delta_{\ii}(\kk) - \delta_{\ii}(\jj)$, $\jj \in \T_n^d$) in place of $f$, we obtain for any $0 \le t \le u$,
\begin{equation}\label{mom21}
\mathbb{E}\left(M^2_{f_{\ii}}(t)\right) - \mathbb{E}\left(M^2_{f_{\ii}}(0)\right) = \int_0^t \sum_{\jj \in \T_n^d}(\Delta p(u-s, \ii + \jj))^2\mathbb{E}(\Delta^2_{\jj}(s))ds.
\end{equation}
Consider the smoothing process $\YY = \YY^{(\zero)}$ on $\T_n^d$.
Note that from the definition, 
\begin{align*}
M_{f_{\ii}}(u) &= \sum_{\jj \in \T_n^d}\Delta p(0, \ii + \jj)Y_{\jj}(u) = \sum_{\jj \in \T_n^d}\left(\frac{1}{2d}\sum_{\kk \sim \jj}\delta_{\ii}(\kk) - \delta_{\ii}(\jj)\right)Y_{\jj}(u)\\
&= \frac{1}{2d}\sum_{\jj \sim \ii}Y_{\jj}(u) - Y_{\ii}(u) = \Delta_{\ii}(u),
\end{align*}
and
$$
M_{f_{\ii}}(0) = \sum_{\jj \in \T_n^d}\Delta p(u, \ii + \jj)Y_{\jj}(0) = \sum_{\jj \in \T_n^d}\Delta p(u, \ii + \jj)\delta_{\zero}(\jj) = \Delta p(u, \ii).
$$
Using these observations and taking $t=u$ in \eqref{mom21},
\begin{equation}\label{mom22}
\mathbb{E}\left(\Delta^2_{\ii}(u)\right) - \left(\Delta p(u, \ii)\right)^2 = \int_0^u \sum_{\jj \in \T_n^d}(\Delta p(u-s, \ii + \jj))^2\mathbb{E}(\Delta^2_{\jj}(s))ds.
\end{equation}
Now, summing over $\ii$ in \eqref{mom22}, we obtain
\begin{align*}
\mathbb{E}&\sum_{\ii \in \T_n^d}\Delta^2_{\ii}(u) - G(u) = \sum_{\ii \in \T_n^d}\int_0^u \sum_{\jj \in \T_n^d}(\Delta p(u-s, \ii + \jj))^2\mathbb{E}(\Delta^2_{\jj}(s))ds\\
&= \int_0^u \sum_{\jj \in \T_n^d}\sum_{\ii \in \T_n^d}(\Delta p(u-s, \ii + \jj))^2\mathbb{E}(\Delta^2_{\jj}(s))ds = \int_0^u \sum_{\jj \in \T_n^d}G(u-s)\mathbb{E}(\Delta^2_{\jj}(s))ds\\
& = \int_0^u G(u-s)\mathbb{E}\left(\sum_{\jj \in \T_n^d}\Delta^2_{\jj}(s)\right)ds.
\end{align*}
Note that the above holds for every $u \ge 0$. Thus the function $H(u) = \mathbb{E}\sum_{\jj \in \T_n^d}\Delta^2_{\jj}(u)$, $u \ge 0,$ solves the renewal equation
\begin{equation*}
H(u) = G(u) + \int_0^u G(u-s)H(s)ds.
\end{equation*}
As $G(\,\cdot\,)$ is non-negative and $\sup_{t \ge 0}G(t) < \infty$, by \cite[Thm. 2]{Feller}, the above equation has a unique non-negative solution which is bounded on every finite time interval, which thus equals $H$. The function $H$ can be expressed in the series form (see \cite[(7.1)]{Feller}) as
$$
H(u) = \sum_{k=1}^{\infty}G^{\star k}(u), u \ge 0,
$$
proving the lemma.
\end{proof}

In view of the above lemma, we need to study the asymptotic properties of the function $G$. 

\begin{lemma}\label{kernel}
There exist positive constants $\beta_1,\beta_2$ and $n_0 \ge 3$ (all depending on $d$) such that for all $n \ge n_0$, $t \ge 4d$,
\begin{align}\label{d16.3}
\frac{\beta_1}{t^{\frac{d}{2} + 2} \wedge n^{d+4}} 
\exp\left(-2\sg_1 d^{-1} t\right) \le G(t) 
\le \frac{\beta_2}{t^{\frac{d}{2} + 2} \wedge n^{d+4}} 
\exp\left(-2\sg_1 d^{-1} t\right).
\end{align}
\end{lemma}
\begin{proof}
We will first study the case $d=1$. 
Recall the notation and spectral representation from Remark \ref{d19.4}.
Note that $\sg_1 = \hl_1$ so we will prove \eqref{d16.3} with $\hl_1$ in place of $\sg_1$.

It follows from \eqref{d19.5}, \eqref{d19.2}
and \eqref{l2grad} that for $t \ge 0$,
\begin{equation}\label{Gform}
G(t) = \sum_{\ii \in \T_n^1} (\Delta p(t,\ii))^2 = \sum_{\ii \in \T_n^1} (\partial_t p(t,\ii))^2 = \frac{2}{n}\sum_{j=1}^{\lfloor\frac{n}{2}\rfloor}\hl_j^2e^{-2\hl_j t}.
\end{equation}

Using \eqref{d11.5} in \eqref{Gform}, we get
\begin{align}\label{n28.1}
G(t) &\ge \frac{2C_1^2}{n}\sum_{j=1}^{\lfloor\frac{n}{2}\rfloor}\frac{j^4}{n^4}e^{-2C_2j^2 t / n^2}= 2C_1^2t^{-5/2}\frac{1}{nt^{-1/2}}\sum_{j=1}^{\lfloor\frac{n}{2}\rfloor}\frac{j^4}{(nt^{-1/2})^4}e^{-2C_2j^2  / (nt^{-1/2})^2},\\
G(t) &\le \frac{2C_2^2}{n}\sum_{j=1}^{\lfloor\frac{n}{2}\rfloor}\frac{j^4}{n^4}e^{-2C_1j^2 t / n^2} = 2C_2^2t^{-5/2}\frac{1}{nt^{-1/2}}\sum_{j=1}^{\lfloor\frac{n}{2}\rfloor}\frac{j^4}{(nt^{-1/2})^4}e^{-2C_1j^2  / (nt^{-1/2})^2}.\label{n28.2}
\end{align}
Fix some $\eps \in (0,1)$. If $4 \le t \le \eps n^2$ then $1/ (nt^{-1/2})\leq \eps^{1/2}$ and the Riemann sum approximation with ``$\Delta x= 1/ (nt^{-1/2})$''  yields for large $n$
\begin{align*}
\frac{1}{nt^{-1/2}}\sum_{j=1}^{\lfloor\frac{n}{2}\rfloor}\frac{j^4}{(nt^{-1/2})^4}e^{-2C_1j^2  / (nt^{-1/2})^2}
\approx
\int_0^{\sqrt{t}/2} x^4 e^{-2C_1 x ^2}dx,
\end{align*}
and similarly for the analogous expression in \eqref{n28.2}.
This observation and \eqref{n28.1}-\eqref{n28.2} imply that there exist  positive constants $C_3, C_4$  such that for all $n,t$ satisfying $4 \le t \le \eps n^2$,
\begin{equation}\label{Gest1}
C_3 t^{-5/2} \le G(t) \le C_4 t^{-5/2}.
\end{equation}
This and \eqref{eigbd} establish the lemma for $d=1$ when $4 \le t \le \eps n^2$.

To address the region $t \ge \eps n^2$, observe that \eqref{Gform} yields
\begin{equation}\label{Gest2}
2n^{-5}(n^2\hl_1)^2e^{-2\hl_1t} \le G(t) 
\le \frac{2\hl_1^2e^{-2\hl_1 t} + 2\hl_2^2e^{-2\hl_2 t}}{n} 
+ \frac{2e^{-2\hl_1t}}{n}\sum_{j=3}^{\lfloor\frac{n}{2}\rfloor}\hl_j^2e^{-2(\hl_j -\hl_1)t}
\end{equation}
where the sum above is taken to be zero if $n=3,4,5$. 
Since $\hl_1 \le \hl_2$ for $n \ge 3$, \eqref{d11.5}  implies that
\begin{align}\label{d16.2}
\frac{2\hl_1^2e^{-2\hl_1 t} + 2\hl_2^2e^{-2\hl_2 t}}{n}
\leq \frac{2(C_2  /n^2)^2e^{-2\hl_1 t} + 2(C_2 2^2 /n^2)^2e^{-2\hl_1 t}}{n}
=\frac{34C_2^2e^{-2\hl_1 t}}{n^5}.
\end{align}
The explicit form of $\hl_j$ and 
\eqref{d16.1} show that for all $ 3 \le j \le n/2$,
\begin{align*}
\hl_j - \hl_1 &= (1-\cos(2 \pi j/n)) - (1- \cos(2\pi/n))\\
&\geq \left(\frac 12\cdot \frac {4 \pi^2 j^2}{n^2}  
-\frac 1 {24} \cdot \frac {16 \pi^4 j^4}{n^4}\right)
-  \frac 12\cdot \frac {4 \pi^2 }{n^2}  
=  \frac {2 \pi^2 j^2}{n^2}\left(1 - \frac 1 {j^2}  
-\frac { \pi^2 j^2}{3n^2}\right)\\
& \geq  \frac {2 \pi^2 j^2}{n^2}\left(1 - \frac 1 {3^2}  
-\frac { \pi^2 (n/2)^2}{3n^2}\right)
> 0.06 \frac {2 \pi^2 j^2}{n^2} =: \delta j^2/n^2.
\end{align*}
It follows from this, \eqref{d11.5}, \eqref{Gest2} and \eqref{d16.2} that if $t \ge \eps n^2$ then
\begin{equation}\label{Gest3}
\frac{(2 C_1)^2e^{-2\hl_1t}}{n^5} \le
\frac{(2n^2\hl_1)^2e^{-2\hl_1t}}{n^5} \le G(t) \le \frac{34C_2^2e^{-2\hl_1 t}}{n^5} + \frac{2C_2^2e^{-2\hl_1t}}{n^5}\sum_{j=3}^{\infty}j^4e^{-2\delta j^2 \eps},
\end{equation}
which establishes the lemma for $d=1$ when $t \ge \eps n^2$.
Thus, \eqref{Gest1} and \eqref{Gest3} together prove the lemma for $d=1$.

Next we will show that there exist positive constants $C_5, C_6$ and $n_1 \ge 3$ such that for all $n \ge n_1$, $t \ge 4$,
\begin{equation}\label{l2fn}
\frac{C_5}{t^{1/2} \wedge n} \le \sum_{\ii \in \T_n^1} p(t,\ii)^2 \le \frac{C_6}{t^{1/2} \wedge n}.
\end{equation}
Note from \eqref{d19.2} that
\begin{align*}
\sum_{\ii \in \T_n^1} \left(p(t,\ii) - \frac{1}{n}\right)^2 
&= \frac{2}{n}\sum_{j=1}^{\lfloor\frac{n}{2}\rfloor}e^{-2\hl_j t },
\end{align*}
so \eqref{d11.5} yields
\begin{align}\label{n28.3}
\sum_{\ii \in \T_n^1} \left(p(t,\ii) - \frac{1}{n}\right)^2 &\ge \frac{2}{n}\sum_{j=1}^{\lfloor\frac{n}{2}\rfloor}e^{-2C_2j^2 t / n^2}= t^{-1/2}\frac{2}{nt^{-1/2}}\sum_{j=1}^{\lfloor\frac{n}{2}\rfloor}e^{-2C_2j^2  / (nt^{-1/2})^2},\\
\sum_{\ii \in \T_n^1} \left(p(t,\ii) - \frac{1}{n}\right)^2 &\le \frac{2}{n}\sum_{j=1}^{\lfloor\frac{n}{2}\rfloor}e^{-2C_1j^2 t / n^2}= t^{-1/2}\frac{2}{nt^{-1/2}}\sum_{j=1}^{\lfloor\frac{n}{2}\rfloor}e^{-2C_1j^2  / (nt^{-1/2})^2}.\label{heatsq}
\end{align}
Fix $\eps \in (0,1)$. If $4 \le t \le \eps n^2$ then $1/ (nt^{-1/2})\leq \eps^{1/2}$ and the Riemann sum approximation yields
\begin{align}\label{n28.4}
\frac{2}{nt^{-1/2}}\sum_{j=1}^{\lfloor\frac{n}{2}\rfloor}e^{-2C_2j^2  / (nt^{-1/2})^2}
\geq C_7
2\int_0^{t/2}  e^{-2C_2 x ^2}dx
\geq C_7
2\int_0^{2}  e^{-2C_2 x ^2}dx.
\end{align}
A similar upper bound holds for the analogous expression in \eqref{heatsq}.
The estimates \eqref{n28.3}-\eqref{n28.4} and the following formula 
\begin{equation}\label{kerbreak}
\sum_{\ii \in \T_n^1} p(t,\ii)^2 = \frac{1}{n} + \sum_{\ii \in \T_n^1} \left(p(t,\ii) - \frac{1}{n}\right)^2,
\end{equation}
imply that for some positive constants $C_8, C_9$ and for all $n,t$ satisfying $4 \le t \le \eps n^2$,
\begin{equation}\label{l2fn1}
C_8 t^{-1/2} \le \sum_{\ii \in \T_n^1} p(t,\ii)^2 \le C_9 t^{-1/2}.
\end{equation}
For $t \ge \eps n^2$, using \eqref{kerbreak} and the first inequality in \eqref{heatsq}, we obtain
\begin{equation}\label{l2fn2}
\frac{1}{n} \le \sum_{\ii \in \T_n^1} p(t,\ii)^2 \le \frac{1}{n} + \frac{2}{n}\sum_{j=1}^{\infty} e^{-2C_1\eps j^2}.
\end{equation}
The bounds \eqref{l2fn} now follow from \eqref{l2fn1} and \eqref{l2fn2}.

Consider $d \ge 2$. In the rest of the proof, we will use $p(t, \x)$ to denote the heat kernel for $\T_n^d$ and  $p^{(1)}(t,\ii)$ to denote the heat kernel for $\T_n^1$. Observe that 
$
p(t,\x) = \Pi_{k=1}^d p^{(1)}(td^{-1}, x_k)
$
for $t \ge 0$ and  $\x = (x_1,\dots, x_d) \in \T_n^d$.
Hence, we can write 
$$
\partial_tp(t,\x) = d^{-1}\sum_{k=1}^d g_k(t,\x)
$$
where 
$$
g_k(t,\x) = \partial_tp^{(1)}(td^{-1}, x_k)\prod_{j\neq k} p^{(1)}(td^{-1}, x_j).
$$
To obtain the upper bound, observe that, by the Cauchy-Schwarz inequality,
\begin{align}\label{lbd}
\sum_{\x \in \T_n^d}(\partial_tp(t,\x))^2 
&\leq
\sum_{\x \in \T_n^d}d^{-1} \sum_{k=1}^d g_k(t,\x)^2 \\
&=
d^{-1}\sum_{\x \in \T_n^d} \sum_{k=1}^d 
\left(
\left(\partial_tp^{(1)}(td^{-1}, x_k)\right)^2
\prod_{j\neq k} p^{(1)}(td^{-1}, x_j)^2
\right) \notag\\
&= \left(\sum_{\ii \in \T_n^1}(\partial_tp^{(1)}(td^{-1}, \ii))^2\right)\left(\sum_{\ii \in \T_n^1}p^{(1)}(td^{-1}, \ii)^2\right)^{d-1}.\notag
\end{align}
By \eqref{Gform} and the upper bound in \eqref{d16.3} for $d=1$,
\begin{align}\label{d16.4}
\sum_{\ii \in \T_n^1}(\partial_tp^{(1)}(td^{-1}, \ii))^2\leq
\frac{\beta_2}{(td^{-1})^{\frac{1}{2} + 2} \wedge n^{5}} 
\exp\left(-2\hl_1  td^{-1}\right).
\end{align}
Note that the first two equalities in  \eqref{Gform} hold for  $\T_n^d$ with any $d\geq 1$. Hence, it follows from \eqref{lbd}, \eqref{d16.4} and the upper bound in \eqref{l2fn}  that there exist $C_{10}>0, n_2\ge 3$ such that for all $n \ge n_2$, $t \ge 4d$,
\begin{align}\label{d1}
G(t)& =
\sum_{\x \in \T_n^d}(\partial_tp(t,\x))^2 
\leq 
\frac{\beta_2}{(td^{-1})^{\frac{1}{2} + 2} \wedge n^{5}} 
\exp\left(-2\hl_1 d^{-1} t\right)
\left(  \frac{C_6}{(td^{-1})^{1/2} \wedge n}\right)^{d-1}\\
&\le \frac{C_{10}}{t^{\frac{d}{2}+2} \wedge n^{d+4}}\exp\left(-2\hl_1 d^{-1} t\right).
\notag
\end{align}
This proves the upper bound in \eqref{d16.3} for $d\geq 2$.

To obtain the lower bound, note that
$$
(\partial_tp(t,\x))^2 = d^{-2}\sum_{k=1}^dg_k(t,\x)^2 + d^{-2}\sum_{k\neq j}g_k(t,\x)g_j(t,\x).
$$
For any $k \neq j$,
\begin{multline*}
\sum_{\x \in \T_n^d}g_k(t,\x)g_j(t,\x)\\
= \left(\sum_{\ii \in \T_n^1}\left(\partial_tp^{(1)}(td^{-1}, \ii)\right)p^{(1)}(td^{-1}, \ii)\right)^2\left(\sum_{\ii \in \T_n^1}p^{(1)}(td^{-1}, \ii)^2\right)^{d-2} \ge 0,
\end{multline*}
so
\begin{align}\label{d17.1}
(\partial_tp(t,\x))^2 \geq d^{-2}\sum_{k=1}^dg_k(t,\x)^2 .
\end{align}

By \eqref{Gform} and the lower bound in \eqref{d16.3} for $d=1$,
\begin{align}\label{d17.2}
\sum_{\ii \in \T_n^1}(\partial_tp^{(1)}(td^{-1}, \ii))^2\geq
\frac{\beta_1}{(td^{-1})^{\frac{1}{2} + 2} \wedge n^{5}} 
\exp\left(-2\hl_1  td^{-1}\right).
\end{align}
It follows from the last two equalities of \eqref{lbd}, along with \eqref{d17.1}, \eqref{d17.2} and the lower bound in \eqref{l2fn}  that there exist $C_{11}>0, n_3\ge 3$ such that for all $n \ge n_3$, $t \ge 4d$,
\begin{align}\label{d1}
G(t)& 
\geq \sum_{\x \in \T_n^d} d^{-2}\sum_{k=1}^dg_k(t,\x)^2\\
&\geq 
\frac 1d\cdot
\frac{\beta_1}{(td^{-1})^{\frac{1}{2} + 2} \wedge n^{5}} 
\exp\left(-2\hl_1 d^{-1} t\right)
\left(  \frac{C_5}{(td^{-1})^{1/2} \wedge n}\right)^{d-1}\nonumber\\
&\geq \frac{C_{11}}{t^{\frac{d}{2}+2} \wedge n^{d+4}}\exp\left(-2\hl_1 d^{-1} t\right).
\notag
\end{align}
This proves the lower bound in \eqref{d16.3} for $d\geq 2$.
\end{proof}

\begin{lemma}\label{d22.2}
$\int_0^{\infty}G(t)dt = 1/2$.
\end{lemma}
\begin{proof}
Recall that we write $\x = (x_1, \dots, x_n)$ for $\x \in \T_n^d$.
By Remark \ref{d19.4}, especially \eqref{d19.5} and \eqref{d19.1},
\begin{align*}
\int_0^{\infty}G(t)dt 
&= \int_0^{\infty}\sum_{\ii \in \T_n^d}(\partial_tp(t,\ii))^2 dt 
= \int_0^{\infty}\frac{1}{n^d}\sum_{\x \in \T_n^d}\left(\hl_{\x}^{(d)}\right)^2e^{-2\hl_{\x}^{(d)}t}dt\\
&= \frac{1}{2n^d}\sum_{\x \in \T_n^d}\hl_{\x}^{(d)}
= \frac{1}{2dn^d}\sum_{\x \in \T_n^d}\sum_{j=1}^d(1-\cos(2\pi x_j/n))\\
&= \frac{1}{2dn^d}dn^{d-1}\sum_{k=0}^{n-1}(1-\cos(2\pi k/n))
=\frac{1}{n}\sum_{k=1}^{n-1}\sin^2(\pi k/n) = \frac{1}{2},
\end{align*}
where the last equality follows from the  trigonometric identity $\sum_{k=1}^{n-1}\sin^2(\pi k/n) = n/2$ (see \cite[(1.18), p. 29]{Lawler}). This completes the proof of the lemma.
\end{proof}

In the remaining part of the article $C$ will denote a generic positive constant (depending on $d$ but not $n$ or $t$) whose value can change from line to line (and even within one line). 

\begin{lemma}\label{laplem}
Let $\beta_2$ be the constant in Lemma \ref{kernel}. There exists a positive integer $n_0 \ge 3$ such that for all $n \ge n_0$,
$$
\int_0^{\infty}
\exp\left(\left(2\sg_1 d^{-1} - \frac{4\beta_2}{n^{d+4}}\right)s\right)
G(s)ds \le \frac{15}{16}.
$$
\end{lemma}

\begin{proof}
For any $\eps\in(0,1)$, we can write
\begin{align}\label{lapt1}
\int_0^{\infty}&
\exp\left(\left(2\sg_1 d^{-1} - \frac{4\beta_2}{n^{d+4}}\right)s\right)
G(s)ds \\
&= \left(\int_0^{\eps n^2}
 + \int_{\eps n^2}^{n^2}+ \int_{n^2}^{\infty}\right)
\exp\left(\left(2\sg_1 d^{-1} - \frac{4\beta_2}{n^{d+4}}\right)s\right)
G(s)ds.\notag
\end{align}

By \eqref{eigbd}, we can choose and fix $\eps>0$ sufficiently small such that for all $n \ge 3$,
$$
\exp\left(\left(2\sg_1 d^{-1} - \frac{4\beta_2}{n^{d+4}}\right)\eps n^2\right)
 \le 9/8.
$$
With this choice of $\eps$, using Lemma \ref{d22.2}, we obtain
\begin{equation}\label{lapt2}
\int_0^{\eps n^2}
\exp\left(\left(2\sg_1 d^{-1} - \frac{4\beta_2}{n^{d+4}}\right)s\right)
G(s)ds \le \frac{9}{8}\int_0^{\infty}G(s)ds = \frac{9}{16}.
\end{equation}
Moreover, using the upper bound in \eqref{d16.3},
\begin{align*}
\int_{\eps n^2}^{n^2}
&\exp\left(\left(2\sg_1 d^{-1} - \frac{4\beta_2}{n^{d+4}}\right)s\right)
G(s)ds \\
&\le \int_{\eps n^2}^{n^2} 
\exp\left(\left(2\sg_1 d^{-1} - \frac{4\beta_2}{n^{d+4}}\right)s\right)
\frac{\beta_2}{s^{\frac{d}{2} + 2} \wedge n^{d+4}}
\exp\left(-2\sg_1 d^{-1} s\right)ds\\
& \le \int_{\eps n^2}^{n^2} \frac{\beta_2}{s^{\frac{d}{2} + 2} \wedge n^{d+4}} ds
= \int_{\eps n^2}^{n^2} \frac{\beta_2}{s^{\frac{d}{2} + 2}}ds \le \frac{C}{(\eps n^2)^{\frac{d}{2} + 1}}.
\end{align*}
Choosing $n_0 \ge 3$ large enough, for any $n \ge n_0$,
\begin{equation}\label{lapt3}
\int_{\eps n^2}^{n^2}
\exp\left(\left(2\sg_1 d^{-1} - \frac{4\beta_2}{n^{d+4}}\right)s\right)
G(s)ds \le \frac{1}{8}.
\end{equation}
Again using the upper bound in \eqref{d16.3},
\begin{align}\label{lapt4}
\int_{n^2}^{\infty}
&\exp\left(\left(2\sg_1 d^{-1} - \frac{4\beta_2}{n^{d+4}}\right)s\right)
G(s)ds \\
&\le \int_{n^2}^{\infty}
\exp\left(\left(2\sg_1 d^{-1} - \frac{4\beta_2}{n^{d+4}}\right)s\right)
\frac{\beta_2}{s^{\frac{d}{2} + 2} \wedge n^{d+4}}
\exp\left(-2\sg_1 d^{-1} s\right)ds
\notag\\
&\le \int_{n^2}^{\infty}\frac{\beta_2}{n^{d+4}}
\exp\left(-\frac{4\beta_2}{n^{d+4}} s\right)ds = \frac{1}{4}
\exp\left(-\frac{4\beta_2 n^2}{n^{d+4}}\right) \le \frac{1}{4}.\notag
\end{align}
Using the bounds \eqref{lapt2}, \eqref{lapt3} and \eqref{lapt4} in \eqref{lapt1}, we obtain $n_0 \ge 3$ such that for all $n \ge n_0$,
\begin{equation}\label{lapt}
\int_0^{\infty}
\exp\left(\left(2\sg_1 d^{-1} - \frac{4\beta_2}{n^{d+4}}\right)s\right)
G(s)ds \le \frac{9}{16} + \frac{1}{8} + \frac{1}{4} = \frac{15}{16},
\end{equation}
proving the lemma.
\end{proof}

\begin{lemma}\label{enderlem}
Let $\beta_2$ be the constant in Lemma \ref{kernel}. There exist positive constants $\theta_1, \theta_2$ and $n_0 \ge 3$ (all depending on $d$) such that for all $n \ge n_0$, $t \ge 4d$,
\begin{align}\label{enderbd}
\frac{\theta_1}{t^{\frac{d}{2} + 2} \wedge n^{d+4}}\exp\left(-2\sg_1 d^{-1} t\right) & \le \mathbb{E}\sum_{\ii \in \T_n^d}\Delta_{\ii}^2(t)\\
& \le \frac{\theta_2}{t^{\frac{d}{2} + 2} \wedge n^{d+4}}\exp\left(-\left(2\sg_1 d^{-1} - \frac{4\beta_2}{n^{d+4}}\right) t\right).\notag
\end{align}
\end{lemma}

\begin{proof}
By Lemma \ref{mom2delta}, for any $t \ge 0$,
$$
\mathbb{E}\sum_{\ii \in \T_n^d}\Delta_{\ii}^2(t)  = \sum_{k=1}^{\infty}G^{\star k}(t) \ge G(t).
$$
Choosing $n_0$ as in Lemma \ref{kernel}, for $n \ge n_0$ and $t \ge 4d$,
$$
\mathbb{E}\sum_{\ii \in \T_n^d}\Delta_{\ii}^2(t) \ge \frac{\beta_1}{t^{\frac{d}{2} + 2} \wedge n^{d+4}} 
\exp\left(-2\sg_1 d^{-1} t\right)
$$
proving the lower bound in \eqref{enderbd}.

Next we will prove the upper bound in \eqref{enderbd}. Define 
\begin{align*}
\wt{G}(t) =
\exp\left(\left(2\sg_1 d^{-1} - \frac{4\beta_2}{n^{d+4}}\right)t\right)
G(t), \quad t \ge 0.
\end{align*}
It follows from \eqref{n28.2} that for all $t>0$,
\begin{align*}
G(t) &\le \frac{2C_2^2}{n}\sum_{j=1}^{\lfloor\frac{n}{2}\rfloor}\frac{j^4}{n^4}e^{-2C_1j^2 t / n^2} 
\le \frac{2C_2^2}{n}\sum_{j=1}^{\lfloor\frac{n}{2}\rfloor}\frac{j^4}{n^4}
\leq C < \infty.
\end{align*}
It follows from this and the upper bound in Lemma \ref{kernel} that there exists a positive constant $C_*$ depending only on $d$, such that
\begin{equation}\label{cbd1}
\wt{G}(t) \le \ol G(t) := \frac{C_*}{t^{\frac{d}{2} + 2} \wedge n^{d+4}}, \ t > 0.
\end{equation}
Write $\theta : = \frac{15}{16}$. By Lemma \ref{laplem},
\begin{equation}\label{cbd0}
\int_0^{\infty}\wt{G}(s)ds \le \theta.
\end{equation}
Define the sequence 
$$
a_m := \frac{(16/17)^{\frac{2(m-1)}{d+4}}}{2\sum_{j=1}^{\infty}(16/17)^{\frac{2(j-1)}{d+4}}}, \ m \ge 1.
$$
As $\sum_{j=1}^{\infty}a_j = 1/2$,
$$
a^* := \prod_{l=1}^{\infty}(1-a_l) >0.
$$
We claim the following: for all $k \ge 1$ and $t>0$,
\begin{equation}\label{convbd}
\wt{G}^{\star k}(t) \le \wh{G}_k(t) := 2\theta^k\ol G\left(t - t\sum_{j=1}^ka_j\right) + \theta^{k-1}\sum_{j=1}^k \ol G\left(ta^*a_j\right).
\end{equation}
Since $\sum_{j=1}^{\infty}a_j = 1/2$, $\wh{G}_k(\,\cdot\,)$ is well-defined and non-increasing for each $k \ge 1$.

We will prove \eqref{convbd} by induction. We now formulate the induction statement. Define $\{b_{jk}\}_{1 \le j \le k< \infty}$ by $b_{jj} := 1$ and
$$
b_{jk} := \prod_{l=j+1}^k(1-a_l), \qquad 1 \le j < k < \infty.
$$
Then we claim that for $k \ge 1$ and $t>0$,
\begin{equation}\label{ind}
\wt{G}^{\star k}(t) \le \wh{G}^*_k(t) := 2\theta^k\ol G\left(t - t\sum_{j=1}^ka_j\right) + \theta^{k-1}\sum_{j=1}^k \ol G\left(tb_{jk}a_j\right).
\end{equation}
As $b_{jk} \ge a^*$ for all $1 \le j \le k < \infty$ and $\ol G(\cdot)$ is non-increasing, establishing \eqref{ind} proves \eqref{convbd} for any $k \ge 1$.

For later reference, we note that $\wh{G}^*_k(\cdot)$ is non-increasing.

The claim \eqref{ind} for $k=1$ follows from \eqref{cbd1} and the observation that $\ol G(\cdot)$ is non-increasing. Suppose the claim \eqref{ind} is true for all $k \le N$. Then for any $t >0$,
\begin{align}\label{cbd2}
\wt{G}^{\star (N+1)}(t) & = \int_0^t \wt{G}(s)\wt{G}^{\star N}(t-s)ds\\
& = \int_0^{ta_{N+1}} \wt{G}(s)\wt{G}^{\star N}(t-s)ds + \int_{ta_{N+1}}^t \wt{G}(s)\wt{G}^{\star N}(t-s)ds.\notag
\end{align}
By the induction hypothesis, the monotonicity of $\wh{G}^*_N(\,\cdot\,)$ and \eqref{cbd0},
\begin{align}\label{cbd3}
\int_0^{ta_{N+1}} \wt{G}(s)\wt{G}^{\star N}(t-s)ds & \le \int_0^{ta_{N+1}} \wt{G}(s)\wh{G}^*_N(t-s)ds\\
& \le \wh{G}^*_N(t-ta_{N+1})\int_0^{ta_{N+1}} \wt{G}(s)ds \le \theta\wh{G}_N^*(t-ta_{N+1}). \notag
\end{align}
From \eqref{cbd1}, \eqref{cbd0}, and using the monotonicity of $\overline{G}(\,\cdot\,)$,
\begin{align}\label{cbd4}
\int_{ta_{N+1}}^t \wt{G}(s)\wt{G}^{\star N}(t-s)ds & \le \overline{G}(ta_{N+1})\int_{ta_{N+1}}^t \wt{G}^{\star N}(t-s)ds\\
& \le \overline{G}(ta_{N+1})\int_{0}^{\infty} \wt{G}^{\star N}(s)ds\notag\\
& = \overline{G}(ta_{N+1})\left(\int_{0}^{\infty} \wt{G}(s)ds\right)^N \le \theta^N \overline{G}(ta_{N+1}).\notag
\end{align}
Using \eqref{cbd3} and \eqref{cbd4} in \eqref{cbd2}, and recalling the explicit form of $\wh{G}^*_N$ given in \eqref{ind}, we obtain
\begin{align*}
\wt{G}^{\star (N+1)}&(t)  \le \theta\wh{G}^*_N(t-ta_{N+1}) + \theta^N \overline{G}(ta_{N+1})\\
& = 2\theta^{N+1}\ol G\left(\left(t - ta_{N+1}\right)\left(1 - \sum_{j=1}^Na_j\right)\right)\\
&\qquad \qquad \qquad + \theta^{N}\sum_{j=1}^N \ol G\left(t(1-a_{N+1})b_{jN}a_j\right) + \theta^N \overline{G}(ta_{N+1})\\
& = 2\theta^{N+1}\ol G\left(t - t\sum_{j=1}^{N+1}a_j + ta_{N+1}\sum_{j=1}^Na_j\right) + \theta^{N}\sum_{j=1}^{N+1} \ol G\left(tb_{j(N+1)}a_j\right)\\
& \le 2\theta^{N+1}\ol G\left(t - t\sum_{j=1}^{N+1}a_j\right) + \theta^{N}\sum_{j=1}^{N+1} \ol G\left(tb_{j(N+1)}a_j\right),
\end{align*}
where we used the observation that $b_{(N+1)(N+1)} = 1$ and $b_{j(N+1)} = (1-a_{N+1})b_{jN}$ for $j \le N$ to obtain the second equality, and the monotonicity of $\overline{G}(\,\cdot\,)$ to obtain the last inequality. This proves the claim \eqref{ind} for $k=N+1$. Thus, by induction, \eqref{ind}, and hence \eqref{convbd}, is true for all $k \ge 1$.

Now we prove the upper bound in \eqref{enderbd} using \eqref{convbd}. 
By Lemma \ref{mom2delta} and \eqref{convbd}, for any $t >0$,
\begin{align}\label{ender1}
\exp&\left(\left(2\sg_1 d^{-1} - \frac{4\beta_2}{n^{d+4}}\right)t\right) \ \mathbb{E}\sum_{\ii \in \T_n^d}\Delta_{\ii}^2(t) \\
& = \exp\left(\left(2\sg_1 d^{-1} - \frac{4\beta_2}{n^{d+4}}\right)t\right)
\sum_{k=1}^{\infty}G^{\star k}(t) 
= \sum_{k=1}^{\infty}\wt{G}^{\star k}(t) \notag \\
&\le \sum_{k=1}^{\infty}\wh{G}_k(t) 
= 2\sum_{k=1}^{\infty}\theta^k\ol G\left(t - t\sum_{j=1}^ka_j\right) + \sum_{k=1}^{\infty}\theta^{k-1}\sum_{j=1}^k \ol G\left(t a^* a_j\right) \notag\\
&\le 2\overline{G}(t/2)\sum_{k=1}^{\infty}\theta^k + \sum_{j=1}^{\infty}\ol G\left(t a^* a_j\right)\sum_{k=j}^{\infty} \theta^{k-1}\notag\\ 
&= \frac{2\theta}{1-\theta}\overline{G}(t/2) + \frac{1}{1-\theta}\sum_{j=1}^{\infty}\theta^{j-1} \ol G\left(t a^* a_j\right)\notag,
\end{align}
where we used the observations that $\overline{G}(\,\cdot\,)$ is non-increasing and $\sum_{j=1}^{\infty}a_j = 1/2$, along with an interchange of summation, to obtain the second inequality.

Recall that $\theta = 15/16$ and note that $\frac{15}{16} \cdot \frac{17}{16}<1$. Consider the case when $t \le 2n^2$ and note that $t/2 \leq n^2$ and $t a^* a_j \leq n^2$. We have, using \eqref{cbd1} in \eqref{ender1},
\begin{align}\label{ender2}
\exp&\left(\left(2\sg_1 d^{-1} - \frac{4\beta_2}{n^{d+4}}\right)t\right)  \mathbb{E}\sum_{\ii \in \T_n^d}\Delta_{\ii}^2(t) \\
& \le \frac{2C_*\theta}{(1-\theta)}\frac{1}{t^{\frac{d}{2} + 2}} + \frac{C_*}{1-\theta}\sum_{j=1}^{\infty}\theta^{j-1} \frac{1}{(ta^*a_j)^{\frac{d}{2} + 2}}\notag\\
&= \frac{2C_*\theta}{(1-\theta)}\frac{1}{t^{\frac{d}{2} + 2}}\notag\\
&\qquad + \frac{C_*}{(a^*)^{\frac{d}{2} + 2}(1-\theta)}\left(2\sum_{j=1}^{\infty}(16/17)^{\frac{2(j-1)}{d+4}}\right)^{\frac{d}{2} + 2} \frac{1}{t^{\frac{d}{2} + 2}}\sum_{j=1}^{\infty}\left(\frac{15}{16}\right)^{j-1} \left(\frac{17}{16}\right)^{j-1} \notag\\
&\le \frac{C}{t^{\frac{d}{2} + 2}}.\notag
\end{align}

For $t > 2n^2$, again using \eqref{cbd1} in \eqref{ender1},
\begin{align}\label{ender3}
\exp&\left(\left(2\sg_1 d^{-1} - \frac{4\beta_2}{n^{d+4}}\right)t\right)  \mathbb{E}\sum_{\ii \in \T_n^d}\Delta_{\ii}^2(t)\\
 & \le \frac{2C_*\theta}{(1-\theta)}\frac{1}{n^{d+4}} + \frac{C_*}{1-\theta}\sum_{j=1}^{\infty}\theta^{j-1} \frac{1}{(ta^*a_j)^{\frac{d}{2} + 2} \wedge n^{d+4}}\notag\\
& = \frac{2C_*\theta}{(1-\theta)}\frac{1}{n^{d+4}} + \frac{C_*}{1-\theta}\sum_{j: ta^*a_j > n^2}\theta^{j-1} \frac{1}{(ta^*a_j)^{\frac{d}{2} + 2} \wedge n^{d+4}}\notag\\
&\qquad + \frac{C_*}{1-\theta}\sum_{j: ta^*a_j \le n^2}\theta^{j-1} \frac{1}{(ta^*a_j)^{\frac{d}{2} + 2} \wedge n^{d+4}}\notag\\
&= \frac{2C_*\theta}{(1-\theta)}\frac{1}{n^{d+4}} + \frac{C_*}{1-\theta}\frac{1}{n^{d+4}}\sum_{j: ta^*a_j > n^2}\theta^{j-1}
+ \frac{C_*}{1-\theta}\sum_{j: ta^*a_j \le n^2}\theta^{j-1} \frac{1}{(ta^*a_j)^{\frac{d}{2} + 2}}\notag\\
&\le \frac{2C_*\theta}{(1-\theta)}\frac{1}{n^{d+4}} + \frac{C_*}{1-\theta}\frac{1}{n^{d+4}}\sum_{j=1}^{\infty}\theta^{j-1}\notag\\
&\qquad + \frac{C_*}{(a^*)^{\frac{d}{2} + 2}(1-\theta)}\left(2\sum_{j=1}^{\infty}(16/17)^{\frac{2(j-1)}{d+4}}\right)^{\frac{d}{2} + 2} \frac{1}{t^{\frac{d}{2} + 2}}\sum_{j=1}^{\infty}\left(\frac{15}{16}\right)^{j-1} \left(\frac{17}{16}\right)^{j-1}\notag\\
&\le \frac{C}{n^{d+4}} + \frac{C}{t^{\frac{d}{2} + 2}} \le \frac{C}{n^{d+4}}\notag.
\end{align}
Combining \eqref{ender2} and \eqref{ender3}, we obtain
$$
\mathbb{E}\sum_{\ii \in \T_n^d}\Delta_{\ii}^2(t) \le \frac{C}{t^{\frac{d}{2} +2} \wedge n^{d+4}}
\exp\left(-\left(2\sg_1 d^{-1} - \frac{4\beta_2}{n^{d+4}}\right)t
\right),
$$
which gives the upper bound in \eqref{enderbd}. This proves the lemma.
\end{proof}

Recall the energy functional $\Et(t)$ defined in \eqref{d19.8}, the generator $\LL$ of the smoothing process on $\T_n^d$ defined in \eqref{d19.7}, and $\Delta_\ii(t)$ defined in \eqref{d19.9}.
\begin{lemma}\label{energen}
For any $t \ge 0$, $\LL \Et(t) = -n^{-d}\sum_{\ii \in \T_n^d}\Delta_{\ii}^2(t)$, and, consequently,
$$
\mathbb{E}\Et(t) = n^{-d}\int_t^{\infty}\mathbb{E}\left(\sum_{\ii \in \T_n^d}\Delta_{\ii}^2(s)\right)ds.
$$
\end{lemma}
\begin{proof}
We have
\begin{align}\label{eg1}
\LL \Et(t) &= \frac{1}{2dn^d} \sum_{\ii \in \T_n^d} \sum_{\jj :\jj \sim \ii}\left[(Y_{\ii}(t) - Y_{\jj}(t) + \Delta_{\ii}(t))^2 - (Y_{\ii}(t) - Y_{\jj}(t))^2\right]\\
 &= \frac{1}{2dn^d} \sum_{\ii \in \T_n^d} \sum_{\jj :\jj \sim \ii}\left[\Delta^2_{\ii}(t) + 2(Y_{\ii}(t) - Y_{\jj}(t))\Delta_{\ii}(t)\right]\nonumber\\
 &= \frac{1}{n^d} \sum_{\ii \in \T_n^d}\Delta^2_{\ii}(t) + \frac{1}{dn^d} \sum_{\ii \in \T_n^d} \sum_{\jj :\jj \sim \ii}(Y_{\ii}(t) - Y_{\jj}(t))\Delta_{\ii}(t)\nonumber.
\end{align}
Note that
\begin{align*}
\frac{1}{dn^d} \sum_{\ii \in \T_n^d} \sum_{\jj :\jj \sim \ii}(Y_{\ii}(t) - Y_{\jj}(t))\Delta_{\ii}(t) &= \frac{1}{dn^d} \sum_{\ii \in \T_n^d}\Delta_{\ii}(t) \sum_{\jj :\jj \sim \ii}(Y_{\ii}(t) - Y_{\jj}(t))\\
& = -\frac{2}{n^d}\sum_{\ii \in \T_n^d}\Delta^2_{\ii}(t).
\end{align*}
Using this in \eqref{eg1},
\begin{align}\label{engeneq}
\LL \Et(t) = -\frac{1}{n^d}\sum_{\ii \in \T_n^d}\Delta^2_{\ii}(t).
\end{align}
Taking expectation and integrating from $t$ to infinity on both sides, we obtain the result.
\end{proof}

We will write $\langle \,\cdot\,, \,\cdot\,\rangle$ to denote the inner product with respect to the counting measure on $\T_n^d$. 
Recall eigenvalues $\hl^{(d)}_\x$ and eigenfunctions $\psi_\x$ from Section \ref{prelim}. In the following, it will be more convenient to relabel the eigenvalues as $0=\hld_0< \hld_1 \leq \dots \leq \hld_{n^d-1}$. The eigenfunctions will be similarly labeled $\psi_j$. Write $V_j(t) = n^{-d/2} \langle \psi_j, \Y(t) \rangle$, $t \ge 0$, $j \ge 1$. Note that $\Y(t) = \sum_{j=0}^{n^d-1}n^{d/2}V_j(t)\psi_j$.

\begin{lemma}
We have
\begin{align}\label{d23.1}
\EE(t) &=  n^{-d}\langle \Y(t), (I-P)\Y(t)\rangle 
= \sum_{j=1}^{n^d-1}\hld_j V^2_j(t),\\
\label{energy2}
\LL \EE(t) &= -n^{-d}\sum_{\ii \in \T_n^d}\Delta_{\ii}^2(t) 
= -n^{-d}\langle \Y(t), (I-P)^2\Y(t)\rangle 
= -\sum_{j=1}^{n^d-1}\left(\hld_j\right)^2 V_j^2(t),\\
\label{energy3}
\Vt(t) &= \left\langle \sum_{j=1}^{n^d-1}V_j(t) \psi_j, \sum_{j=1}^{n^d-1}V_j(t) \psi_j \right\rangle = \sum_{j=1}^{n^d-1}V_j^2(t),\\
\label{vgen}
\LL \Vt(t) &= -2\EE(t) - (1-n^{-d})\LL \EE(t).
\end{align}
\end{lemma}

\begin{proof}
The first equality in \eqref{d23.1} follows from \eqref{d19.8} and the observation that
$$
\frac{1}{4dn^d}\sum_{\ii \sim \jj}(Y_{\ii}(t) - Y_{\jj}(t))^2=n^{-d}\langle \Y(t), (I-P)\Y(t)\rangle.
$$ 
Since $\Y(t) = \sum_{j=0}^{n^d-1}n^{d/2}V_j(t)\psi_j$ and $(I-P)\Y(t) = \sum_{j=1}^{n^d-1}n^{d/2}\hld_j V_j(t)\psi_j$, the second equality in \eqref{d23.1} follows.

Note that $\Delta_{\ii}(t) = ((P-I)\Y(t))_{\ii}$ for $\ii \in \T_n^d$ and hence 
$$
n^{-d}\sum_{\ii \in \T_n^d}\Delta_{\ii}^2(t) = n^{-d}\langle (P-I)\Y(t), (P-I)\Y(t)\rangle = n^{-d}\langle \Y(t), (I-P)^2\Y(t)\rangle
$$
for $t \ge 0$. Thus, using \eqref{engeneq}, 
\begin{equation*}
\LL \EE(t) = -n^{-d}\sum_{\ii \in \T_n^d}\Delta_{\ii}^2(t) 
= -n^{-d}\langle \Y(t), (I-P)^2\Y(t)\rangle 
= -\sum_{j=1}^{n^d-1}\left(\hld_j\right)^2 V_j^2(t),
\end{equation*}
where the third equality follows by noting that $$(I-P)^2\Y(t) = \sum_{j=1}^{n^d-1}\left(\hld_j\right)^2 V_j(t)\psi_j$$ for $t \ge 0$. This proves \eqref{energy2}. 

The first equality in \eqref{energy3} follows from \eqref{n26.1}-\eqref{d2.1} and the observation that
$$
n^{-d}\sum_{\ii \in \T_n^d}\left(Y_{\ii}(t) - \overline{Y}(t)\right)^2 = \left\langle \sum_{j=1}^{n^d-1}V_j(t) \psi_j, \sum_{j=1}^{n^d-1}V_j(t) \psi_j \right\rangle.
$$
The second equality in \eqref{energy3} follows from the orthonormality of $\{\psi_{j}(\cdot) : 0 \le j \le n^{d} -1\}$.

In view of \eqref{d19.9},
the second term  in \eqref{lyap5} (with $\yy$ replaced by $\Y(t)$) can be written as
$$
n^{-2d}\langle (I-P)\Y(t), (I-P)\Y(t)\rangle = n^{-2d}\langle \Y(t), (I-P)^2\Y(t)\rangle = n^{-2d}\sum_{\ii \in \T_n^d}\Delta_{\ii}^2(t).
$$
Hence, by \eqref{lyap5} and \eqref{d10.3},
\begin{equation}\label{vgenstep}
\LL \Vt(t) = -\EE^{(2)}(t) - n^{-2d}\sum_{\ii \in \T_n^d}\Delta_{\ii}^2(t),
\end{equation}
where $\EE^{(2)}(t) := n^{-d}\langle \Y(t), (I-P^2)\Y(t)\rangle$ is the two-step energy functional.
Note that
\begin{align*}
\EE^{(2)}(t) &= n^{-d}\langle \Y(t), (I-P^2)\Y(t)\rangle \\
&= 2n^{-d}\langle \Y(t), (I-P)\Y(t)\rangle 
- n^{-d}\langle \Y(t), (I-P)^2\Y(t)\rangle\\ 
&= 2n^{-d}\langle \Y(t), (I-P)\Y(t)\rangle - n^{-d}\sum_{\ii \in \T_n^d}\Delta_{\ii}^2(t) = 2\EE(t) + \LL \EE(t),
\end{align*}
where the last equality follows from \eqref{engeneq}.
Thus, using the above observation and \eqref{engeneq} in \eqref{vgenstep}, the variance functional has the following generator:
\begin{equation*}
\LL \Vt(t) = -2\EE(t) - (1-n^{-d})\LL \EE(t).
\end{equation*}
\end{proof}

\begin{proof}[Proof of Theorem \ref{wass}]
In view of \eqref{eigbd}, if $ t_0 \le t \le n^2$ then 
$0 \le n^{-2}t/C \le \sg_1d^{-1} t \leq C n^{-2} t \le C$ and, therefore,
\begin{align}\label{d22.6}
1/C\le\exp\left(-\sg_1d^{-1} t\right) \leq 1.
\end{align}
Recall $\beta_2$ from Lemma \ref{kernel}.
We use \eqref{eigbd} to find $C'$ and $n_1$, depending only on $d$, such that for $n\geq n_1$,
\begin{align}\label{j25.1}
2\sg_1 d^{-1} - \frac{4\beta_2}{n^{d+4}} > C' n^{-2} >0.
\end{align}
In the rest of the proof we will always assume that $n_0\geq n_1$.

Using Lemma \ref{energen}, the lower bound in \eqref{enderbd} and \eqref{d22.6}, for $t_0 \le t \le n^2$,
\begin{align*}
n^d\mathbb{E}\Et(t) &=
\int_t^{\infty}\mathbb{E}\left(\sum_{\ii \in \T_n^d}\Delta_{\ii}^2(s)\right)ds 
\ge \int_t^{\infty}\frac{C}{s^{\frac{d}{2} + 2} \wedge n^{d+4}} 
\exp\left(-2\sg_1d^{-1} s\right)ds\notag\\
& \ge \exp\left(-4\sg_1 d^{-1} n^2\right)\int_t^{2n^2}\frac{C}{s^{\frac{d}{2} + 2}} ds 
 \ge C\int_t^{2t}\frac{C}{s^{\frac{d}{2} + 2}} ds 
 \ge C t^{-(\frac{d}{2} + 1)}.\notag
\end{align*}
This yields for $n \ge n_0$, $ t_0 \le t \le n^2$,
\begin{align}\label{wass3}
\mathbb{E}\Et(t) & \ge C n^{-d} t^{-(\frac{d}{2} + 1)}
\ge \frac{C n^{-d}}{t^{\frac{d}{2} + 1} \wedge n^{d+2}} 
\exp\left(-2\sg_1d^{-1} t\right).
\end{align}
For $t > n^2$, using Lemma \ref{energen}, the lower bound in \eqref{enderbd} and \eqref{eigbd},
\begin{align}\label{wass4}
n^d\mathbb{E}\Et(t) &=
\int_t^{\infty}\mathbb{E}\left(\sum_{\ii \in \T_n^d}\Delta_{\ii}^2(s)\right)ds 
\ge \int_t^{\infty} \frac{C}{n^{d+4}} \exp\left(-2\sg_1d^{-1} s\right)ds\\
&= \frac{C}{\sg_1n^{d+4}} \exp\left(-2\sg_1d^{-1} t\right) \ge \frac{C}{n^{d+2}} \exp\left(-2\sg_1d^{-1} t\right)\notag\\
&= \frac{C }{t^{\frac{d}{2} + 1} \wedge n^{d+2}} 
\exp\left(-2\sg_1d^{-1} t\right).\notag
\end{align}
Now we obtain the upper bound in the energy estimate \eqref{d12.1}. For $n \ge n_0$, $ t_0 \le t \le n^2$, using Lemma \ref{energen}, the upper bound in \eqref{enderbd},  and \eqref{j25.1},
\begin{align}\label{wass4.1}
n^d\mathbb{E}\Et(t) &=
\int_t^{\infty}\mathbb{E}\left(\sum_{\ii \in \T_n^d}\Delta_{\ii}^2(s)\right)ds\\ 
& \le \int_t^{\infty}\frac{C}{s^{\frac{d}{2} + 2} \wedge n^{d+4}} 
\exp\left(-\left(2\sg_1 d^{-1} - \frac{4\beta_2}{n^{d+4}}\right) s\right)ds\nonumber\\
&\le \exp\left(-\left(2\sg_1 d^{-1} - \frac{4\beta_2}{n^{d+4}}\right) t\right)\int_t^{n^2}\frac{C}{s^{\frac{d}{2} + 2}}ds\notag\\
&\qquad + \int_{n^2}^{\infty}\frac{C}{n^{d+4}}\exp\left(-\left(2\sg_1 d^{-1} - \frac{4\beta_2}{n^{d+4}}\right) s\right)ds\notag\\
&\le \frac{C}{t^{\frac{d}{2} + 1}}\exp\left(-\left(2\sg_1 d^{-1} - \frac{4\beta_2}{n^{d+4}}\right) t\right)\notag\\
&\qquad + \frac{C}{n^{d+4}\left(2\sg_1 d^{-1} - \frac{4\beta_2}{n^{d+4}}\right)}\exp\left(-\left(2\sg_1 d^{-1} - \frac{4\beta_2}{n^{d+4}}\right) n^2\right)\notag\\
&\le \frac{C}{t^{\frac{d}{2} + 1}}\exp\left(-\left(2\sg_1 d^{-1} - \frac{4\beta_2}{n^{d+4}}\right) t\right)\nonumber\\
&\qquad + \frac{C}{n^{d+2}}\exp\left(-\left(2\sg_1 d^{-1} - \frac{4\beta_2}{n^{d+4}}\right) n^2\right)\notag\\
&\le \frac{C}{t^{\frac{d}{2} + 1}}\exp\left(-\left(2\sg_1 d^{-1} - \frac{4\beta_2}{n^{d+4}}\right) t\right).\notag
\end{align} 
For $t \ge n^2$, again using Lemma \ref{energen}, the upper bound in \eqref{enderbd} and \eqref{j25.1},
\begin{align}\label{wass4.2}
n^d\mathbb{E}\Et(t) &=
\int_t^{\infty}\mathbb{E}\left(\sum_{\ii \in \T_n^d}\Delta_{\ii}^2(s)\right)ds\\ 
& \le \int_t^{\infty}\frac{C}{s^{\frac{d}{2} + 2} \wedge n^{d+4}} 
\exp\left(-\left(2\sg_1 d^{-1} - \frac{4\beta_2}{n^{d+4}}\right) s\right)ds\nonumber\\
&= \int_t^{\infty}\frac{C}{n^{d+4}} 
\exp\left(-\left(2\sg_1 d^{-1} - \frac{4\beta_2}{n^{d+4}}\right) s\right)ds\notag\\
&= \frac{C}{n^{d+4}\left(2\sg_1 d^{-1} - \frac{4\beta_2}{n^{d+4}}\right)}\exp\left(-\left(2\sg_1 d^{-1} - \frac{4\beta_2}{n^{d+4}}\right) t\right)\notag\\
&\le \frac{C}{n^{d+2}}\exp\left(-\left(2\sg_1 d^{-1} - \frac{4\beta_2}{n^{d+4}}\right) t\right).\notag
\end{align}
The energy bounds in \eqref{d12.1} with $\alpha_3 := 4\beta_2$ thus follow from \eqref{wass3}, \eqref{wass4}, \eqref{wass4.1} and \eqref{wass4.2}.

Now we obtain the variance bounds in \eqref{d12.2}.
To obtain the lower bound in \eqref{d12.2}, note that by the Cauchy-Schwarz inequality, \eqref{d23.1}, \eqref{energy2} and \eqref{energy3}, for any $t \ge 0$,
\begin{align}\label{Nash1}
\Et(t) &= \sum_{j=1}^{n^d-1}\hld_j V_j^2(t) \le \left(\sum_{j=1}^{n^d-1}\left(\hld_j\right)^2 V_j^2(t)\right)^{1/2}\left(\sum_{j=1}^{n^d-1} V_j^2(t)\right)^{1/2}\\
 &= \left(n^{-d}\sum_{\ii \in \T_n^d}\Delta_{\ii}^2(t)\right)^{1/2}\Vt(t)^{1/2}.\notag
\end{align}
Taking expectations in \eqref{Nash1} and applying the Cauchy-Schwarz inequality again, we obtain
\begin{align*}
\mathbb{E}\Et(t)& \le \mathbb{E}\left[\left(n^{-d}\sum_{\ii \in \T_n^d}\Delta_{\ii}^2(t)\right)^{1/2}\Vt(t)^{1/2}\right] 
\le \left(n^{-d}\mathbb{E}
\sum_{\ii \in \T_n^d}\Delta_{\ii}^2(t)
\right)^{1/2}
\left(\mathbb{E}\Vt(t)\right)^{1/2}.
\end{align*}
Hence, for any $t \ge 0$,
\begin{equation}\label{vlb}
\mathbb{E}\Vt(t) \ge \frac{\left(\mathbb{E}\Et(t)\right)^2}{n^{-d}\mathbb{E}
\sum_{\ii \in \T_n^d}\Delta_{\ii}^2(t)}.
\end{equation}
Using the lower bound in \eqref{d12.1} and the upper bound in \eqref{enderbd} in \eqref{vlb}, we obtain for $n \ge n_0$, $t \ge t_0$,
\begin{align}\label{varl}
\mathbb{E}\Vt(t) 
&\ge  \frac{\displaystyle\frac{Cn^{-2d}}{t^{d+2} \wedge n^{2d+4}}
\exp\left(-4\sg_1 d^{-1} t\right)}
{\displaystyle\frac{Cn^{-d}}{t^{\frac{d}{2} + 2} \wedge n^{d+4}}\exp\left(-\left(2\sg_1 d^{-1} - \frac{4\beta_2}{n^{d+4}}\right) t\right)}\\
& = \frac{Cn^{-d}}{t^{\frac{d}{2}} \wedge n^{d}}\exp\left(-\left(2\sg_1 d^{-1} + \frac{4\beta_2}{n^{d+4}}\right) t\right).\notag
\end{align}
This is the lower bound in \eqref{d12.2}.

Using the energy bounds \eqref{d12.1} in place of the bounds on $\mathbb{E}\left(\sum_{\ii \in \T_n^d}\Delta_{\ii}^2(t)\right)$ obtained in \eqref{enderbd}, calculations similar to  \eqref{wass4.1} and \eqref{wass4.2} show that there exist positive constants $C_2, t_0$ and $n_0 \ge 3$ (depending on $d$) such that for all $n \ge n_0$, $t \ge t_0$,
\begin{align}\label{j15.2}
 \int_{t}^{\infty}\mathbb{E}\Et(s)ds \le \frac{C_2n^{-d}}{t^{\frac{d}{2}} \wedge n^{d}}\exp\left(-\left(2\sg_1 d^{-1} - \frac{4\beta_2}{n^{d+4}}\right) t\right).
\end{align}
To obtain the upper bound for the variance in \eqref{d12.2}, we use \eqref{vgen} and  \eqref{j15.2} to obtain for $n \ge n_0$, $t \ge t_0$,
\begin{align}\label{varu}
\mathbb{E}\Vt(t) 
&= 2\int_{t}^{\infty}\mathbb{E}\Et(s)ds - (1-n^{-d}) \mathbb{E}\Et(t)\\
&\leq 2\int_{t}^{\infty}\mathbb{E}\Et(s)ds \leq \frac{2C_2n^{-d}}{t^{\frac{d}{2}} \wedge n^{d}}\exp\left(-\left(2\sg_1 d^{-1} - \frac{4\beta_2}{n^{d+4}}\right) t\right).\notag
\end{align}
The bounds in \eqref{d12.2} with $\alpha_3 := 4\beta_2$ follow from \eqref{varl} and \eqref{varu}. This proves the theorem.
\end{proof}

\begin{theorem}\label{locsmooth}
Fix $\ii \in \T_n^d$. Recall that  $\{\YY^{(\ii)}(t) : t \ge 0\}$ denotes the smoothing process with starting configuration $Y^{(\ii)}_{\xx}(0) = \delta_{\ii}(\xx)$ for $\xx \in \T_n^d$. Recall the constants $\alpha_1, \alpha_2$ and $\alpha_3$ from Theorem \ref{wass}.

(i)
There exist  constants  $n_0 \ge 3$, $t_0 >0$ (depending on $d$) such that for all $n \ge n_0$, $t \ge t_0$,
\begin{align*}
\frac{\alpha_1}{t^{\frac{d}{2} + 1} \wedge n^{d+2}} 
\exp\left(-2\sg_1d^{-1} t\right)  
&\le \mathbb{E}\left[
\left(\sum_{\jj \in \T_n^d} Y^{(\ii)}_{\jj}(t) - |\T_n^d | \  Y^{(\ii)}(\infty)\right)^2\right] \\
&\le \frac{\alpha_2}{t^{\frac{d}{2} + 1} \wedge n^{d + 2}} 
\exp\left(-\left(2\sg_1 d^{-1} - \frac{\alpha_3}{n^{d+4}}\right) t\right).
\end{align*}

(ii) Consider a family of non-negative random variables $\{X_{\jj} : \jj \in \mathbb{Z}^d\}$, independent of $\YY^{(\ii)}$, for which $\operatorname{Cov}\left(X_{\jj}, X_{\kk}\right) = 0$ for $\jj \neq \kk$ and $\mathbb{E}(X_{\jj}) = \mu < \infty$, $\operatorname{Var}(X_{\jj}) = \sigma^2 \in (0,\infty)$ for all $\jj \in \mathbb{Z}^d$. There exist  $\eta_1>0$ and $\eta_2>0$ (depending on $\mu$, $\sigma$ and $d$), and $n'_0 \ge 3$ and $t'_0 >0$ (depending on $d$ only) such that for all $n \ge n'_0$, $t \ge t'_0$,
\begin{align*}
\frac{\eta_1}{t^{\frac{d}{2}} \wedge n^{d}}\exp\left(-\left(2\sg_1 d^{-1} + \frac{\alpha_3}{n^{d+4}}\right) t\right) 
&\le  \mathbb{E}\left[\left(\sum_{\jj \in \T_n^d} X_{\jj}Y^{(\ii)}_{\jj}(t) - Y^{(\ii)}(\infty)\sum_{\jj \in \T_n^d} X_{\jj}\right)^2 \right]\\
&\le \frac{\eta_2}{t^{\frac{d}{2}} \wedge n^{d}}\exp\left(-\left(2\sg_1 d^{-1} - \frac{\alpha_3}{n^{d+4}}\right) t\right).
\end{align*}

\end{theorem}
\begin{proof}
(i) Let $U(t) = \sum_{\jj \in \T_n^d} Y^{(\ii)}_{\jj}(t)$. Taking $f \equiv 1$ in Lemma \ref{heatmart}, it follows that $U$ is a martingale with predictable quadratic variation given by
$
\langle U\rangle (t) = \int_0^t\sum_{\jj \in \T_n^d} \Delta^2_{\jj}(s)ds.
$
Hence, by Lemma \ref{energen},
$$
\mathbb{E}(U(t) - U(\infty))^2 = \int_t^{\infty}\mathbb{E}\left(\sum_{\jj \in \T_n^d}\Delta_{\jj}^2(s)\right)ds = n^{d}\mathbb{E}\Et(t).
$$
The result now follows from \eqref{d12.1}.

(ii) Note that
\begin{align*}
\mathbb{E}&\left[\left(\sum_{\jj \in \T_n^d} X_{\jj}Y^{(\ii)}_{\jj}(t) - Y^{(\ii)}(\infty)\sum_{\jj \in \T_n^d} X_{\jj}\right)^2\right]\\
& = \sigma^2\mathbb{E}\left(\sum_{\jj \in \T_n^d}(Y^{(\ii)}_{\jj}(t) - Y^{(\ii)}(\infty))^2\right) 
+ \mu^2 \mathbb{E}\left[\left(\sum_{\jj \in \T_n^d} Y^{(\ii)}_{\jj}(t) - |\T_n^d | \  Y^{(\ii)}(\infty)\right)^2\right]\\
& = \sigma^2\mathbb{E}\left(\sum_{\jj \in \T_n^d}(Y^{(\ii)}_{\jj}(t) - \overline{Y^{(\ii)}}(t))^2\right)\\ 
&\qquad + (\mu^2+ \sigma^2n^{-d}) \mathbb{E}\left[
\left(\sum_{\jj \in \T_n^d} Y^{(\ii)}_{\jj}(t) - |\T_n^d | \  Y^{(\ii)}(\infty)\right)^2\right]\\
&  = \sigma^2n^d \mathbb{E}\Vt(t) + (\mu^2+ \sigma^2n^{-d}) \mathbb{E}\left[ \left(\sum_{\jj \in \T_n^d} Y^{(\ii)}_{\jj}(t) - |\T_n^d | \  Y^{(\ii)}(\infty)\right)^2\right].
\end{align*}
The result now follows from part (i) and \eqref{d12.2}.
\end{proof}

\begin{remark}\label{vardif}
The upper bound in part (ii) of Theorem \ref{locsmooth} holds under the weaker assumption that the random variables $\{X_{\jj} : \jj \in \mathbb{Z}^d\}$ satisfy $\operatorname{Cov}\left(X_{\jj}, X_{\kk}\right) = 0$ for $\jj \neq \kk$,  $\mathbb{E}(X_{\jj}) = \mu < \infty$ for all $\jj \in \mathbb{Z}^d$, and $\sigma^2 := \sup_{\jj \in \mathbb{Z}^d}\operatorname{Var}(X_{\jj}) < \infty$. This can be shown by replacing  the first ``$=$'' with ``$\le$'' in the proof of part (ii).
\end{remark}

\begin{proof}[Proof of Theorem \ref{avgquick}]
Recall definitions \eqref{d25.1} and \eqref{d25.2}, Remark \ref{d26.5} and notation from 
Theorem \ref{avgquick}. In particular, recall 
 the collection of Poisson processes $\mathcal{P}$ and the family of processes $\{Y^{(\ii)}(\,\cdot\,), \ii \in \T_n^d\}$. Define $Z^{(\ii)}(\,\cdot\,) = \sum_{\jj \in \T_n^d} Y^{(\ii)}_{\jj}(\,\cdot\,)$, $\ii \in \T_n^d$. Then, using the linearity of the smoothing process, the processes $\YY$ and $\overline{Y}$ have the following representations,
\begin{align*}
Y_{\jj}(\,\cdot\,) = \sum_{\ii \in \T_n^d} y_\ii Y^{(\ii)}_{\jj}(\,\cdot\,) \text{ for } \jj \in \T_n^d,
\qquad 
\overline{Y}(\,\cdot\,) = n^{-d}\sum_{\ii \in \T_n^d}y_{\ii}Z^{(\ii)}(\,\cdot\,).
\end{align*}
 Note that the smoothing process is constant in time if the initial mass at every vertex of the torus is the same. Thus, $n^{-d}\sum_{\ii \in \T_n^d}Z^{(\ii)}(t) = 1$ for all $t\geq 0$. Using this observation along with linearity of the smoothing process, we can write $\overline{Y}(\,\cdot\,) = y^* + n^{-d}\sum_{\ii \in \T_n^d}(y_{\ii} - y^*)Z^{(\ii)}(\,\cdot\,)$. Thus, for any $t \ge 0$, using the Cauchy-Schwarz inequality,
\begin{align*}
\mathbb{E}&\left[\left(\overline{Y}(t) - \overline{Y}(\infty)\right)^2 \right]
= \mathbb{E}\left[\left(n^{-d}\sum_{\ii \in \T_n^d}(y_{\ii} - y^*)(Z^{(\ii)}(t) - Z^{(\ii)}(\infty))\right)^2\right]\\
& \le \mathbb{E}\left[\left(n^{-d}\sum_{\ii \in \T_n^d}|y_{\ii} - y^*|\right)\left(n^{-d}\sum_{\ii \in \T_n^d}|y_{\ii} - y^*|(Z^{(\ii)}(t) - Z^{(\ii)}(\infty))^2\right)\right]\\
& = \mathbb{E}\left[\left(n^{-d}\sum_{\ii \in \T_n^d}|y_{\ii} - y^*|\right)^2 \right]
\mathbb{E}\left[\left(Z^{(\zero)}(t) - Z^{(\zero)}(\infty)\right)^2\right],
\end{align*}
where we have used the independence of $\mathcal{P}$ and the random variables $\{y_{\ii}, \ii \in \T_n^d\}$, and the invariance in law of the process $Z^{(\ii)}(\,\cdot\,)$ with respect to translations of $\ii$ on the torus $\T_n^d$, to obtain the last equality. The theorem now follows from the upper bound in part (i) of Theorem \ref{locsmooth}.
\end{proof}

\begin{proof}[Proof of Theorem \ref{metone}]
The proof of Theorem \ref{meteorconv} 
applies to any finite state space $\mathcal{I}$, so we can apply it to $\mathcal{I} = \T_n^d$. We repeat the steps of that proof
up to and including the first two lines of \eqref{d25.3} to obtain,
\begin{align*}
\left[d^{(2)}_W(\DD(\XX(t)),\DD(\XX(\infty)))\right]^2
& \le \mathbb{E}\sum_{\ii\in \T_n^d}\left| \sum_{\jj\in \T_n^d} X_\jj(0)Y^{(\ii)}_\jj(t) - Y^{(\ii)}(\infty)\sum_{\jj\in \T_n^d} X_\jj(0)\right|^2.
\end{align*}
This, an application of the upper bounds in Theorem \ref{locsmooth} (i) and (ii) and Remark \ref{vardif} yield \eqref{d26.2} and \eqref{d26.4}.

The proof of Theorem \ref{meteorconv} also shows that the following version of \eqref{d25.3} holds for all $\ii$ 
(but note that the role of $n$ in Theorem \ref{meteorconv} is now played by $n^d$),
\begin{align*}
\left[d^{(2)}_W(\DD(X_\ii(t)),\DD(X_\ii(\infty)))\right]^2
& \le \mathbb{E}\left| \sum_{\jj\in \T_n^d} X_\jj(0)Y^{(\ii)}_\jj(t) - Y^{(\ii)}(\infty)\sum_{\jj\in \T_n^d} X_\jj(0)\right|^2.
\end{align*}
Once again, we can combine this observation with an application of the upper bounds in Theorem \ref{locsmooth} (i) and (ii) and Remark \ref{vardif} to obtain \eqref{d26.1} and \eqref{d26.3}.
\end{proof}

\begin{proof}[Proof of Theorem \ref{meteorzd}] The weak convergence claim of the theorem was proved in \cite{liggett1981ergodic}. 

Recall $\{\Y^{(\ii)}(\cdot) : \ii \in \mathbb{Z}^d\}$ and  $\wt \XX(t)$ from Remark \ref{d26.5}. In particular, we have $\wt X_{\ii}(t) = \sum_{j \in \mathbb{Z}^d} X_{\jj}Y^{(\ii)}_{\jj}(t)$, for $t \ge 0$, $\ii \in \mathbb{Z}^d$. 

We will suppress the superscript $\zero$ and write $\YY$ for $\YY^{(\zero)}$. For any odd integer $n \ge 3$, denote by $\YY^{n}$ the smoothing process on the torus, identified with  $\mathbb{B}^d_n$ (defined in \eqref{d27.1}), and coupled in the natural way with $\YY$: use the same Poisson processes of clock rings driving $\YY$ at the sites of $\mathbb{B}^d_n$ to construct $\YY^n$. We assume that $Y^n_\jj(0) = \delta_\zero(\jj)$ for $\jj \in \mathbb{B}^d_n$ and $Y^n_\jj(t) = 0$ if $\jj \notin \mathbb{B}^d_n$ and $t\geq 0$. We will denote the variance and energy functionals of the process $\YY^n$ by $\Vt^n$ and $\EE^n$, respectively (see \eqref{d2.1} and \eqref{d19.8}). 

A standard application of the ``graphical method'' introduced in \cite{Har78} shows that for every $t>0$ there exists a finite random set $\A_t$ such that  for all $\ii\notin \A_t$, a.s.,
\begin{align}\label{d28.2}
\sup_{0\leq s \leq t} Y_{\ii}(s)= 0.
\end{align}
For an implementation of the graphical method that proves a result on the meteor (potlatch) process that is essentially  ``dual'' to our claim, see  \cite[Prop. 2.1]{burdzymeteor}.
It follows that for every $t\geq 0$
there exists a random $K < \infty$ such that $\A_t \subset \mathbb{B}^d_n$ and, therefore, $\Y = \Y^n$, for $n\geq K$. 
Hence, for every $t\geq 0$
 and $\jj \in \mathbb{Z}^d$,
\begin{align}\label{d28.1}
Y_{\jj}^n(t) \to Y_{\jj}(t), \qquad \text{a.s., as } \ n \rightarrow \infty.
\end{align}

Write $\sigma^2 = \sup_{\jj \in \mathbb{Z}^d}\operatorname{Var}(X_{\jj})$. By \eqref{d28.1} and Fatou's Lemma, for any $0 \le s \le t$,
\begin{align}\label{d28.3}
\mathbb{E}&\left[\left(\sum_{\jj \in \mathbb{Z}^d}X_{\jj}(Y_{\jj}(t) - Y_{\jj}(s))\right)^2\right]\\ 
&\le \sigma^2\mathbb{E}\left(\sum_{\jj \in \mathbb{Z}^d}(Y_{\jj}(t) -Y_{\jj}(s))^2\right)
+ \mu^2\mathbb{E}\left[\left(\sum_{\jj \in \mathbb{Z}^d}Y_{\jj}(t) - \sum_{\jj \in \mathbb{Z}^d}Y_{\jj}(s)\right)^2\right]\notag\\
&\le \sigma^2 \liminf_{n \rightarrow \infty} 
 \mathbb{E}\left(\sum_{\jj \in \mathbb{Z}^d}(Y^n_{\jj}(t) -Y^n_{\jj}(s))^2\right)\notag\\
&\qquad + \mu^2 \liminf_{n \rightarrow \infty} 
\mathbb{E}\left[\left(\sum_{\jj \in \mathbb{Z}^d}Y^n_{\jj}(t) - \sum_{\jj \in \mathbb{Z}^d}Y^n_{\jj}(s)\right)^2\right]\notag\\
&\le \sigma^2 \limsup_{n \rightarrow \infty} 
 \mathbb{E}\left(\sum_{\jj \in \mathbb{Z}^d}(Y^n_{\jj}(t) -Y^n_{\jj}(s))^2\right)\notag\\
&\qquad + \mu^2 \limsup_{n \rightarrow \infty} 
\mathbb{E}\left[\left(\sum_{\jj \in \mathbb{Z}^d}Y^n_{\jj}(t) - \sum_{\jj \in \mathbb{Z}^d}Y^n_{\jj}(s)\right)^2\right]\notag\\
&= \sigma^2
\limsup_{n \rightarrow \infty} \mathbb{E}\left(\sum_{\jj \in \mathbb{B}_n^d}(Y^n_{\jj}(t) -Y^n_{\jj}(s))^2\right)\notag\\
&\qquad + \mu^2
\limsup_{n \rightarrow \infty}
\mathbb{E}\left[\left(\sum_{\jj \in \mathbb{B}_n^d}Y^n_{\jj}(t) - \sum_{\jj \in \mathbb{B}_n^d}Y^n_{\jj}(s)\right)^2\right]\notag.
\end{align}
Note that all the sums of the form $\sum_{\jj \in \mathbb{Z}^d}$ in the above formula are in fact finite sums, in view of \eqref{d28.2}.

For any $n \ge 3$ and $0 \le s \le t$, 
\begin{align}\label{d28.4}
\mathbb{E}&\left(\sum_{\jj \in \mathbb{B}_n^d}(Y_{\jj}^n(t) - Y_{\jj}^n(s))^2\right)\\
& \le 3\mathbb{E}\left(\sum_{\jj \in \mathbb{B}_n^d}(Y_{\jj}^n(t) -\overline{Y}^n(t))^2\right)
+ 3\mathbb{E}\left(\sum_{\jj \in \mathbb{B}_n^d}(Y_{\jj}^n(s) -\overline{Y}^n(s))^2\right)\notag\\
&\qquad +3 n^d\mathbb{E}\left[\left(\overline{Y}^n(t) -\overline{Y}^n(s)\right)^2\right]\nonumber\\
&= 3\mathbb{E}\left(\sum_{\jj \in \mathbb{B}_n^d}(Y_{\jj}^n(t) -\overline{Y}^n(t))^2\right)
+ 3\mathbb{E}\left(\sum_{\jj \in \mathbb{B}_n^d}(Y_{\jj}^n(s) -\overline{Y}^n(s))^2\right)\nonumber\\
&\qquad  + 3n^{-d}\mathbb{E}
\left[\left(\sum_{\jj \in \mathbb{B}_n^d}Y_{\jj}^n(t) - \sum_{\jj \in \mathbb{B}_n^d}Y_{\jj}^n(s)\right)^2\right].\nonumber
\end{align}

Recall the constant $\alpha_3$ defined in Theorem \ref{wass}. We use \eqref{j25.1} and the observation that $\sup_{0< x< \infty}\frac{1}{1 \wedge x^{-d/2}} 
\exp\left(-C x\right)< \infty$ to find $n_0 \ge 3$ such that for all $n \ge n_0$, $t >0$,
\begin{align}\label{d27.5}
\frac{\alpha_2 n^{-d}}{t^{\frac{d}{2}} \wedge n^{d}} 
\exp\left(-\left(2\sg_1 d^{-1} - \frac{\alpha_3}{n^{d+4}}\right) t\right)
&\leq
\frac{\alpha_2 n^{-d}}{t^{\frac{d}{2}} \wedge n^{d}} 
\exp\left(-C t/n^2\right)\\
&= t^{-d/2} \frac{\alpha_2 n^{-d}}{1 \wedge (n^2/t)^{d/2}} 
\exp\left(-C t/n^2\right)\notag\\
& \leq C t^{-d/2}  n^{-d}.\notag
\end{align}
This and  the (uniform in $n$) upper bound in  \eqref{d12.2} show that there exist positive constants $C, t_0$ such that for all $t > s \ge t_0$,
\begin{align}\label{zd2}
\limsup_{n \rightarrow \infty} & \left[\mathbb{E}\left(\sum_{\jj \in \mathbb{B}_n^d}(Y_{\jj}^n(t) -\overline{Y}^n(t))^2\right)
+ \mathbb{E}\left(\sum_{\jj \in \mathbb{B}_n^d}(Y_{\jj}^n(s) -\overline{Y}^n(s))^2\right) \right]\\
&=\limsup_{n \rightarrow \infty}\left[n^d\mathbb{E}\Vt ^n(t) + n^d\mathbb{E}\Vt ^n(s)\right] \le Cs^{-d/2}.\notag
\end{align}

By taking $f \equiv 1$ in Lemma \ref{heatmart} and using Lemma \ref{energen}, 
\begin{equation*}
\mathbb{E}\left[\left(\sum_{\jj \in \mathbb{B}_n^d}Y^n_{\jj}(t) - \sum_{\jj \in \mathbb{B}_n^d}Y^n_{\jj}(s)\right)^2 \right]
= n^d\mathbb{E}\EE^n(s) - n^d\mathbb{E}\EE^n(t). 
\end{equation*} 
From this, an estimate analogous to \eqref{d27.5}, and the upper bound in \eqref{d12.1}, we  obtain positive constants $C, t_0$ (without loss of generality, $C,t_0$ can be chosen to be same as for \eqref{zd2}) such that for all $t > s \ge t_0$,
\begin{align}\label{zdnew3}
\limsup_{n \rightarrow \infty} \mathbb{E}
\left[\left(\sum_{\jj \in \mathbb{B}_n^d}Y^n_{\jj}(t) - \sum_{\jj \in \mathbb{B}_n^d}Y^n_{\jj}(s)\right)^2\right]
& = \limsup_{n \rightarrow \infty}\left[n^d\mathbb{E}\EE^n(s) - n^d\mathbb{E}\EE^n(t)\right]\\
& \le Cs^{-\frac{d}{2}-1}.\notag
\end{align}
From \eqref{d28.4}, \eqref{zd2} and \eqref{zdnew3}, we thus conclude for all $t > s \ge t_0$,
\begin{equation}\label{zdnew4}
\limsup_{n \rightarrow \infty} \mathbb{E}\left(\sum_{\jj \in \mathbb{B}_n^d}(Y_{\jj}^n(t) - Y_{\jj}^n(s))^2\right) \le Cs^{-d/2}.
\end{equation}
Thus, using \eqref{d28.3}, \eqref{zdnew3} and \eqref{zdnew4}, we obtain the following bound for all $t > s \ge t_0$,
\begin{equation}\label{zd5}
\mathbb{E}\left[\left(\sum_{\jj \in \mathbb{Z}^d}X_{\jj}(Y_{\jj}(t) - Y_{\jj}(s))\right)^2 \right]
\le C\sigma^2s^{-d/2} + C\mu^2s^{-\frac{d}{2}-1}.
\end{equation}
Recall that $\wt X_{\0}(t) = \sum_{\jj \in \mathbb{Z}^d}X_{\jj}Y_{\jj}(t)$. From \eqref{zd5}, we can obtain an increasing sequence $t_k \rightarrow \infty$ such that
$$
\mathbb{E}\left[\left(\wt  X_{\0}(t_{k+1}) - \wt  X_{\0}(t_k)\right)^2\right]
 \le 2^{-k} \ \text{ for all } k \ge 1.
$$
Thus, $\wt  X_{\0}(t_k)$ is Cauchy in $L^2$ and hence $\wt  X_{\0}(t_k) \rightarrow \wt  X_{\0}(\infty)$ in $L^2$ as $k \rightarrow \infty$. By Theorem 4.1 of \cite{burdzymeteor}, $\wt  X_{\0}(t_k)$ converges weakly to $X_{\0}(\infty)$ and hence, $\wt  X_{\0}(\infty)$ and $X_{\0}(\infty)$ have the same distribution. From \eqref{zd5}, for any $t \ge t_0$ and $k$ such that $t_k \ge t$,
$$
\mathbb{E}\left[\left(\wt  X_{\0}(t) - \wt  X_{\0}(t_k)\right)^2\right]
 \le C\sigma^2t^{-d/2} + C\mu^2t^{-\frac{d}{2}-1}.
$$
Taking $k \rightarrow \infty$ in the above,
\begin{align}\label{d28.7}
\mathbb{E}\left[\left(\wt  X_{\0}(t) - \wt  X_{\0}(\infty)\right)^2\right] \le C\sigma^2t^{-d/2} + C\mu^2t^{-\frac{d}{2}-1}.
\end{align}

Recalling that $\wt  X_{\0}(t), \wt  X_{\0}(\infty)$ have the same distributions as $X_{\0}(t), X_{\0}(\infty)$, respectively, we obtain
$$
\left[d^{(2)}_W(\DD(X_{\0}(t)),\DD(X_{\0}(\infty)))\right]^2 \le  C\sigma^2t^{-d/2} + C\mu^2t^{-\frac{d}{2}-1}.
$$
This proves the first Wasserstein bound of the theorem for $\phi(\xx) = x_{\0}$. The claimed bound for $X_{\jj} \equiv 1$  follows upon taking $\sigma = 0$ above.

Now consider general $\phi$ which is Lipschitz with Lipschitz constant $\rho$ defined in \eqref{d28.8}, and with support contained in $\mathbb{B}_N^d$ for some odd $N  \ge 3$. The constants in the proof of \eqref{d28.7} do not depend on the special choice of coordinate $\0$ so
$$
\sup_{\ii \in \mathbb{Z}^d}\mathbb{E}
\left[\left(\wt  X_{\ii}(t) - \wt  X_{\ii}(\infty)\right)^2\right]
 \le  C\sigma^2t^{-d/2} + C\mu^2t^{-\frac{d}{2}-1}.
$$
Hence,
\begin{align*}
\mathbb{E}\left[\left(\phi(\wt  \XX(t)) - \phi(\wt  \XX(\infty))\right)^2\right]
 &\le \rho^2 \sum_{\ii \in \mathbb{B}_N^d}\mathbb{E}
 \left[\left(\wt  X_{\ii}(t) - \wt  X_{\ii}(\infty)\right)^2\right]\\
&\le \rho^2 N^{d}\left(C\sigma^2t^{-d/2} + C\mu^2t^{-\frac{d}{2}-1}\right),
\end{align*}
from which the bounds for general $\phi$ follow.
\end{proof}

\bibliographystyle{alpha}
\bibliography{NSFbib1}
\end{document}